\documentclass{article}

\usepackage{arxiv}

\usepackage[utf8]{inputenc} 
\usepackage[T1]{fontenc}    
\usepackage{hyperref}       
\usepackage{url}            
\usepackage{booktabs}       
\usepackage{amsfonts}       
\usepackage{nicefrac}       
\usepackage{microtype}      
\usepackage{lipsum}		
\usepackage{graphicx}
\usepackage{natbib}
\usepackage{doi}

\usepackage{amsmath,amsthm,amsfonts,amssymb}
\usepackage{enumerate,color,xcolor}
\usepackage{url}
\usepackage[utf8]{inputenc}

\newcommand{\bc}[1]{{\color{blue} #1}}
\newcommand{\bcred}[1]{{\color{orange} #1}}
\newtheorem{theorem}{Theorem}

\newtheorem{condition}{Condition}

\newtheorem{lemma}[theorem]{Lemma}

\newtheorem{proposition}[theorem]{Proposition}

\newtheorem{remark}{Remark}
\newtheorem{assumption}{Assumption}

\usepackage{todonotes}
\def\sa#1{\textcolor{red}{#1}}
\def\fprod#1{\left\langle#1\right\rangle}

\newcommand{\vertiii}[1]{{\left\vert\kern-0.25ex\left\vert\kern-0.25ex\left\vert #1 
    \right\vert\kern-0.25ex\right\vert\kern-0.25ex\right\vert}}

\newcommand{\mRx}{\mathbb{R}^{d_x}}
\newcommand{\mRy}{\mathbb{R}^{d_y}}
\newcommand{\tPhi}{\tilde{\Phi}}

\newcommand{\tn}{\tilde{\nabla}}

\newcommand{\mcR}{\mathcal{R}}

\newcommand{\xitdt}{\xi^{(\tau,\sigma,\theta)}}
\newcommand{\ptdt}{\pi^{(\tau,\sigma,\theta)}}
\newcommand{\mRd}{\mathbb{R}^d}
\newcommand{\mcRtdt}{\mathcal{R}^{(\tau,\sigma,\theta)}}

\newcommand{\mE}{\mathbb{E}_{\pi^{(\theta)}_*}}
\newcommand{\hes}{\nabla^{(2)}f_*}

\newcommand{\gV}{\cV^{(\tau,\sigma,\theta)}_\alpha}
\def\grad{\nabla}

\def\cB{\mathcal{B}}
\def\cC{\mathcal{C}}
\def\cD{\mathcal{D}}

\def\cO{\mathcal{O}}

\def\cR{\mathcal{R}}

\def\cV{\mathcal{V}}

\def\mE{\mathbb{E}}

\def\smskip{\smallskip}

\def\texitem#1{\par\smskip\noindent\hangindent 25pt
               \hbox to 25pt {\hss #1 ~}\ignorespaces}

\def\norm#1{\|#1\|}

\newcommand{\BEAS}{\begin{eqnarray*}}
\newcommand{\EEAS}{\end{eqnarray*}}
\newcommand{\BEA}{\begin{eqnarray}}
\newcommand{\EEA}{\end{eqnarray}}
\newcommand{\BEQ}{\begin{eqnarray}}
\newcommand{\EEQ}{\end{eqnarray}}
\newcommand{\BIT}{\begin{itemize}}
\newcommand{\EIT}{\end{itemize}}
\newcommand{\BNUM}{\begin{enumerate}}
\newcommand{\ENUM}{\end{enumerate}}

\newcommand{\BA}{\begin{array}}
\newcommand{\EA}{\end{array}}

\newcommand{\reals}{\mathbb{R}}

\newcommand{\Rank}{\mathop{\bf rank}}

\newif\ifpagenumbering
\pagenumberingtrue

\pagenumberingfalse

\newsavebox{\theorembox}
\newsavebox{\lemmabox}
\newsavebox{\defnbox}
\newsavebox{\assbox}
\savebox{\theorembox}{\noindent\bf Theorem}
\savebox{\lemmabox}{\noindent\bf Lemma}
\savebox{\defnbox}{\noindent Definition}

\usepackage{soul}

\usepackage[normalem]{ulem}
\DeclareUnicodeCharacter{2212}{~}

\definecolor{plum}{rgb}{0.3,0,0.7}

\newcommand{\mg}[1]{{\color{plum}{#1}}}

\renewcommand{\mg}[1]{{\color{black}{#1}}}
\renewcommand{\sa}[1]{{\color{black}{#1}}}
\renewcommand{\bc}[1]{{\color{black}{#1}}}
\renewcommand{\bcred}[1]{{\color{black}{#1}}}

\title{A Variance-Reduced Stochastic Accelerated Primal Dual Algorithm}

\author{Bugra Can\\
Department of Management Sciences and Information Systems\\
	Rutgers Business School \\
	Piscataway, NJ, 08854.\\
	\texttt{bc600@scarletmail.rutgers.edu}
	\And
	Mert G\"urb\"uzbalaban \\
	Department of Management Sciences and Information Systems\\
	Rutgers Business School\\
	Piscataway, NJ, 08854 \\
	\texttt{mert.gurbuzbalaban@rutgers.edu}
	\And
	Necdet Serhat Aybat\\
	Industrial and Manufacturing Engineering Department\\
	Penn State University\\
	University Park, PA, 16802-4400,\\
	\texttt{nsa10@psu.edu}
}

\renewcommand{\shorttitle}{\textit{arXiv} Template}

\hypersetup{
pdftitle={A Variance-Reduced Stochastic Accelerated Primal Dual Algorithm},
pdfsubject={q-bio.NC, q-bio.QM},
pdfauthor={B. Can, M. Gurbuzbalaban, N.S. Aybat},
pdfkeywords={Saddle point algorithms, Variance reduction methods},
}

\begin{document}
\maketitle

\begin{abstract}
In this work, we consider strongly convex strongly concave (SCSC) 
\sa{saddle point~(SP)} problems $\min_{x\in\mathbb{R}^{\bc{d_x}}} \max_{y\in\mathbb{R}^{\bc{d_y}}} f(x,y)$ where $f$ is \sa{$L$-smooth, $f(\cdot,y)$ is $\mu$-strongly convex for every $y$, and $f(x,\cdot)$ is $\mu$-strongly concave for every $x$}. Such problems arise frequently in machine learning in the 
\sa{context of robust 
empirical risk minimization~(ERM), e.g., \emph{distributionally robust} ERM, where partial gradients 
are} estimated 
\sa{using} mini-batches of data points. 
\sa{Assuming we have access to an unbiased stochastic first-order oracle, 
we consider the 
stochastic accelerated primal dual (SAPD) 
algorithm 
 recently introduced in \cite{zhang2021robust} for SCSC SP problems as a robust method 
against the} gradient noise. In particular, SAPD recovers the well-known stochastic gradient descent ascent (SGDA) as a special case when the momentum parameter is set to zero and can achieve an accelerated rate when the momentum parameter is properly tuned, \sa{i.e., improving the $\kappa\triangleq L/\mu$ dependence from $\kappa^2$ for SGDA to $\kappa$.} We propose efficient variance-reduction strategies for SAPD based on Richardson-Romberg extrapolation and show that our method improves upon SAPD both in practice and in theory.
\end{abstract}


\section{Introduction}\label{sec: introduction}
\mg{Saddle point (SP) problems arise frequently in many key settings in machine learning. Examples include but are not limited to robust training of machine learning models \cite{gurbuzbalaban2020stochastic,duchi2021learning}, training Generative Adversarial Networks (GANs) \cite{pmlr-v70-arjovsky17a}, design of fair classifiers \cite{nouiehed2019solving}, and empirical risk minimization~(ERM) problems such as regression and classification \cite{palaniappan2016stochastic}. \sa{Furthermore, ERM problems} with constraints on the model parameters result in constrained stochastic optimization problems which can also be cast into saddle-point problems 
\sa{using} Lagrangian duality.} 

\mg{Motivated by such applications, we consider the strongly convex/strongly concave~\sa{(SCSC)} saddle-point~\sa{(SP)} problem
\begin{equation}\label{def: sapd-opt-prob}
    \min_{x\in \mRx}\max_{y\in\mRy} \; f(x,y),
\end{equation}
where $f:\mathbb{R}^{d_x}\times \mathbb{R}^{d_y} \to \mathbb{R}$ is smooth and strongly convex 
\sa{in} $x$ and strongly \sa{concave} 
\sa{in} $y$. As we discuss in the numerical experiments section, our algorithms and results are also directly applicable to the convex/concave setting by adding an appropriate quadratic regularizer. \sa{In a different setting, SCSC problems of the form given in \eqref{def: sapd-opt-prob} arise as subproblems when solving weakly convex-weakly concave SP problems using inexact proximal point method, e.g., see~\cite{liu2021first}}.

\sa{We consider \eqref{def: sapd-opt-prob} assuming we only have} access to unbiased stochastic estimates of the true gradients $\nabla_x f(x,y)$ and $\nabla_y f(x,y)$; as often in machine learning applications gradients are not exactly computed but \sa{are} estimated from randomly sampled subsets (mini-batches) of data.} This is also the case in privacy-preserving empirical risk minimization where noise is added to the gradients to preserve privacy of the user data \cite{chaudhuri2011differentially}. \mg{For solving these SP problems, stochastic first-order (SFO) methods that rely on stochastic gradient information have been very popular in practice due to their favorable scalability properties\sa{; but,} they come with a number of challenges. In particular, even though the problem \eqref{def: sapd-opt-prob} is well-studied when the gradients are deterministic, many aspects of the stochastic setting with inexact gradient information are relatively understudied.
In fact, accelerated SFO methods \sa{achieving} acceleration beyond bilinear problems are quiet recent in the literature, e.g., \cite{accSFO_SP_1, fallah_2020, zhang2021robust, VarRedSGLD,VarRedSpiderBoost,thekumparampil2019efficient,chen2017accelerated,hsieh2019convergence} and the references therein. To that end, the \sa{recently proposed} stochastic accelerated primal-dual \sa{(SAPD) method~\cite{zhang2021robust}, 
\sa{based on momentum-averaging,} 
can achieve the best iteration complexity bound among single-loop methods for SCSC problems of the form: $\min_x\max_yg(x)+f(x,y)-h(y)$, where $f$ is a smooth SCSC function admitting an SFO oracle as we assume for the SP problem in~\eqref{def: sapd-opt-prob}, and $g,h$ are closed, strongly convex functions with efficient proximal maps; furthermore, SAPD achieves the optimal complexity for bilinear SP problems with $f(x,y)=\fprod{Ax,y}$ for some $A\in\reals^{d_y\times d_x}$.} In particular, SAPD, \sa{which is an extension of the APD~\cite{hamedani2021primal} algorithm from deterministic to stochastic first-order oracle setting,} recovers the well-known stochastic gradient descent ascent (SGDA) as a special case when the momentum parameter is set to zero and can achieve an accelerated rate when the momentum parameter is properly tuned, \sa{i.e., improving the $\kappa\triangleq L/\mu$ dependence from $\kappa^2$ for SGDA to $\kappa$.}} 
 

Due to persistent nature of the noise on the gradients, the performance of the stochastic SP methods \bcred{differs} from their deterministic counterparts and \bcred{depends} heavily on 
\mg{the statistical properties of \sa{the limiting point/distribution generated as a solution} to the problem in \eqref{def: sapd-opt-prob}. In particular, addition of a momentum-averaging step has a typical affect to increase the stationary variance of the iterates in the context of stochastic momentum methods such as SAPD \cite{zhang2021robust} and in this context developing efficient variance-reduction strategies becomes a key for achieving better practical performance.}

\mg{
\textbf{Contributions.} In this paper, we consider the SAPD algorithm with constant (primal and dual) stepsize for solving \sa{SCSC} SP problems of the form 
\sa{\eqref{def: sapd-opt-prob}}. Employing constant stepsize has at least two major benefits: (i) only one value needs to be tuned as opposed to \sa{$a,b$ and $c$ parameters needed to be tuned for decaying step size sequence of the form $a / (b+ck)$ for $k\geq 0$;} (ii) the ``bias term"\bcred{, which characterizes how fast \sa{the effects of initial conditions on} the iterates are forgotten,} decays exponentially (see Proposition \ref{prop: conv-of-SAPD}). 

Our contributions are as follows:
\begin{itemize}
    \item In Proposition \ref{prop: conv-of-SAPD}, we consider a parametrization of \sa{the primal stepsize $\tau$ and dual stepsize $\sigma$} in terms of the momentum parameter $\theta$ and obtain convergence guarantees for SAPD for that particular choice of parameters. \sa{For this specific parametrization of $\tau,\sigma$ in $\theta$, working directly with the expected distance square~(EDS) metric, rather than the more stronger expected gap metric as in~\cite{zhang2021robust}, we were able to provide 
    a sharper result for the EDS metric in terms of constants,} compared to those available for SAPD in 
    \cite{zhang2021robust}. \sa{Under this parametrization,} the corresponding linear convergence rate for the bias term is equal to the \sa{momentum} parameter $\theta$ \sa{as also shown in~\cite{zhang2021robust}}. For ill-conditioned problems we expect \sa{a} slower convergence \sa{for the bias as selecting small $\tau,\sigma>0$ requires $\theta \in (0,1)$ chosen close to $1$.}
    \item Under some assumptions on the gradient noise structure, we show that SAPD iterates admit an invariant distribution (Theorem \ref{thm: inv-meas-exists}). Since SAPD iterates result in a Markov chain that is \emph{non-reversible}, standard tools for reversible Markov Chains (such as those arising in the study of stochastic gradient descent methods \cite{bridging-the-gap}) are not directly applicable. We achieve this result by showing that so-called ``minorization and drift conditions" hold for the Markov chain corresponding to the SAPD iterates.
    \item In Lemma \ref{lem: finite-moment}, we characterize the second and fourth moments of the stationary distribution in terms of its dependency to the momentum parameter $\theta$. Building on this lemma, in Theorem~\ref{thm: expansion-at-inv-meas-mean}, we show that the expected iterates contain a bias that is proportional to $\mathcal{O}(1-\theta)$ as $\theta \to 1$. Motivated by this result, \sa{based on Richardson-Romberg extrapolation type arguments,} we propose a variance-reduction scheme: \sa{running SAPD 
    twice independently with different parameters $\theta_1$ and $\theta_2$, one can remove the bias term proportional to $\cO(1-\theta)$ by considering a linear combination of the two iterate sequences corresponding to two appropriately chosen parameters $\theta_1$ and $\theta_2$}. We call this method variance-reduced SAPD (VR-SAPD). Although it is known that Richardson extrapolation can be very efficient for optimization algorithms such as stochastic gradient descent \cite{bridging-the-gap} and Frank-Wolfe methods \cite{bach-richardson}, to our knowledge, our paper provides the first application of this technique to the accelerated stochastic \sa{primal-dual algorithms for saddle-point problems}. There are also variance-reduction techniques for finite-sum problems where the objective is in the form of an average of finitely many component functions, e.g., \cite{palaniappan2016stochastic,chavdarova2019reducing,alacaoglu2021stochastic}; however, these techniques are not available to our setting as our objective is more general and we do not assume that it has a finite-sum structure.
    \item In Theorem \ref{thm: av-seq-conv}, we characterize the gap between \sa{the expected} ergodic average of SAPD iterates and \sa{the expectation of SAPD iterates in the limit} as the number of iterations $k$ grows. 
    \sa{We show} that this gap is of the order $\mathcal{O}(1/k)$ and is vanishing asymptotically; therefore, it implies that 
    \sa{the same variance-reduction mechanism} can be applied to the ergodic averages of the SAPD iterates as well.
    \item We showcase the efficiency of our proposed algorithm VR-SAPD on distributionally robust logistic regression~\sa{(DRLR)} problems on three datasets where we compare VR-SAPD with SAPD \cite{zhang2021robust}, S-OGDA \cite{fallah_2020}, and Stochastic Mirror Prox (SMP) \cite{SMP} algorithms. 
    \sa{For the DRLR problem, we adopted $f_2$-divergence to define the uncertainty set around the uniform distribution; therefore, the problem domain is the Euclidean space, and in this case} SMP reduces to the stochastic gradient descent ascent (SGDA) method which is commonly used in machine learning practice, e.g., 
    for training GANs \cite{kontonis2020convergence}. Our results show that VR-SAPD reduces the variance of the SAPD iterates considerably and \bcred{results} in a significant performance improvement.
\end{itemize}
}

\textbf{Notations.} 
\sa{Let $\mathbb{R}_+$ is the set of positive real numbers.} For any vector $v \in \mathbb{R}^d$, $\Vert x \Vert$ \sa{denotes the Euclidean norm, 
and for any matrix $M\in\mathbb{R}^{d_x\times d_y}$, $\Vert M \Vert_2$ denotes the spectral norm, 
i.e., the largest singular value of $M$.}
Let $E$ and $F$ be two real vector spaces. \mg{We use $E\otimes F$ to denote the tensor product of $E$ and $F$}. \mg{The tensor product of two vectors $x\in E$ and $y\in F$ is denoted as \sa{$E\otimes F \ni x\otimes y =xy^\top$}}.  Similarly, $E^{\otimes k}$ is the k-th tensor power of $E$ and $x^{\otimes k}\in E^{\otimes k}$ is the  k-th tensor product of $x\in E$. \bcred{We use $\mathbb{N}$ for the set of non-negative integers, and $\mathbb{N}_+\triangleq\mathbb{N}\setminus \{0\}$.}
Let $n\in \mathbb{N}_{+}$, then $\bcred{\mathcal{C}^{n}(\mathbb{R}^d,\mathbb{R}^{m})}$ is the set of n-times continuously differentiable functions from $\mathbb{R}^{d} \rightarrow \mathbb{R}^{m}$. Let $f \in \mathcal{C}^{n}(\mathbb{R}^{d},\mathbb{R}^{m})$, \bcred{$D^{n}f$} is
the $n$-th differential of $f$. \bcred{If $f\in\mathcal{C}^{n}(\mathbb{R}^{d},\mathbb{R})$, then  $\bcred{\nabla}^{(n)}f(z)$ denotes the tensor of order $n$ at the point $z\in\mathbb{R}^{d}$ \mg{arising in the Taylor expansion of $f$}.} For any vector $z\in\mathbb{R}^{d}$ and any matrix $M\in\mathbb{R}^{d\times d}$, we define $[\bcred{\nabla}^{(3)}f(\bcred{z})M]_{\ell}=\sum_{i=1}^{d}\sum_{j=1}^dM_{i,j} \frac{\partial^{3}f}{\partial \bcred{z}_{i}\partial \bcred{z}_j\partial \bcred{z}_\ell}(\bcred{z})$.
\bcred{For any matrices $M, N\in \mathbb{R}^{d\times d}$ $M\otimes N$ is defined as the endomorphism of $\mathbb{R}^{d\times d}$ such that $M\otimes N: P \rightarrow \bcred{MPN^\top}$.}
Lastly, $(x_*,y_*)$ is the \sa{unique} solution of the \sa{SP} problem \sa{in} \eqref{def: sapd-opt-prob} and 
$\nabla^{(2)}f_*$ and $\nabla^{(3)}f_*$ are used for $\nabla^{(2)}f(x_*,y_*)$ and $\nabla^{(3)}f(x_*,y_*)$, respectively, for simplicity.
\section{Preliminaries and Background}
\mg{\paragraph{SAPD Algorithm.} Recently, 
Zhang \emph{et al.} \cite{zhang2021robust} have proposed the (SAPD) algorithm} to solve \sa{SCSC SP problems of the form} \eqref{def: sapd-opt-prob}. \sa{SAPD updates are given as follows:} 
\begin{subequations}\label{alg: sapd}
\begin{align}
\tilde{q}_k&= (1+\theta)\tn_y f(x_k,y_k)-\theta \tn_y f(x_{k-1},y_{k-1}),\\
 \sa{y_{k+1}}
&= y_k+\sigma \tilde{q}_k,\\
x_{k+1}&= x_k-\tau \tn_xf(x_k, \sa{y_{k+1}}). 
\end{align}
\end{subequations}
Basically, SAPD admits three \sa{positive} parameters $(\tau,\sigma,\theta)$ where $\tau$ is the primal stepsize, $\sigma$ is the dual stepsize, and $\theta$ is the momentum parameter. 
In our work, we are particularly interested in \sa{designing a method for the SP problem in \eqref{def: sapd-opt-prob}}
 that would improve the variance of \sa{the SAPD iterate sequence $\{(x_k,y_k)\}_{k\geq 0}$. We call our method \emph{Variance Reduced SAPD}~\mg{(VR-SAPD)}.} \mg{For analysis purposes, we 
 \sa{first state} our assumptions on $f$ and \sa{$\tilde\grad f$}.}
\begin{assumption}\label{assump: obj-fun-prop} \sa{We assume $f\in\cC^4$, i.e., four times continuously differentiable} \bc{with uniformly bounded \sa{$3$-rd and $4$-th order} derivatives}. For any $\hat{y}\in \mRy$, \sa{$f(\cdot,\hat{y})$} is \sa{$L_{xx}$-smooth} and 
$\mu_x$-strongly convex for some \sa{$L_{xx},\mu_x>0$}. 
Furthermore, for any $\hat{x}\in\mRx$, 
\sa{$f(\hat{x},\cdot)$} is 
$\mu_y$-strongly concave, \sa{and there exist constants $L_{yx},L_{yy}> 0$ such that
\begin{align*}
&\Vert \nabla_y f(x,y) - \nabla f_y(\hat{x},\hat{y})\Vert \leq L_{yx}\Vert x-\hat{x}\Vert + L_{yy}\Vert y-\hat{y}\Vert,
\end{align*}
for all $x,\hat{x} \in \mRx$ and $y,\hat{y}\in \mRy$.}
\end{assumption}
\sa{We also make the following standard assumption on $\tilde\grad f$, which says that we 
have access to unbiased stochastic estimates of the gradients.}

\begin{assumption}\label{assump: noise-properties} 
\sa{Let $\{x_k,y_k\}$ be the SAPD iterate sequence. For all $k\geq 0$,
we have access to $\tn_y f (x_k,y_k,\sa{\zeta_k^y})$ and $\tn_x f(x_k,y_{k+1},\sa{\zeta_k^x})$ which are unbiased estimates of $\nabla_y f(x_k,y_k)$ and $\nabla_x f(x_k,y_{k+1})$, respectively;\footnote{\sa{$\zeta_k^y$ and $\zeta_k^x$ are determining the noise structure, and for simplicity of the notation we suppress them and use $\tilde\grad_y f(x_k,y_k)$ and $\tilde\grad_x f(x_k,y_{k+1})$ instead.}} i.e.,
\begin{align*}
    &\mathbb{E}[\tn_y f(x_k,y_k,\zeta_k^y)| x_k,y_k]=\nabla_y f(x_k,y_k),\\
    &\mathbb{E}[\tn_x f(x_k,y_{k+1},\zeta_k^x) | x_k,y_{k+1}]=\nabla_x f(x_k,y_{k+1}),
\end{align*}
such that
the $\{\zeta_k^x\}_k$ and $\{\zeta_k^y\}$ are independent sequences and also independent from each other. Furthermore, for $p\in\{2,3,4\}$, 
there exists $\delta_{(p)}>0$ such that
{\small
\begin{subequations}
\label{ineq: noise-var}
\begin{align}
&\mathbb{E}[\Vert \tn_y f(x_k,y_k,\zeta^y_k)-\nabla\bcred{y} f(x_k,y_k)\Vert^{p}] \leq \sa{\delta_{(p)}^p},\\
&\mathbb{E}[\Vert \tn_x f(x_k,y_{k+1},\zeta^x_k)-\nabla_{\bcred{x}} f(x_k,y_{k+1})\Vert^{p}] \leq \sa{\delta_{(p)}^p}.
\end{align}
\end{subequations}}%
}%
\bc{\sa{We also assume that both noise, i.e.,} $\tn_xf-\nabla_xf$ and $\tn_yf-\nabla_yf$, are stationary, and \bcred{independent from the \mg{past}.}} 
\end{assumption}
\mg{We note that $f$ being twice continuously differentiable and existence of second moments for the gradient noise would be sufficient to obtain convergence guarantees \sa{in $L^2$ for the accumulation of SAPD sequence $\{x_k,y_k\}$ in the vicinity of the unique saddle point--see Theorem~\ref{thm:complexity} below}; however, \sa{in this paper}, we require further smoothness up to fourth order to be able to obtain stronger convergence guarantees for the fourth moments which is then used in our variance reduction mechanism (see~Lemma~\ref{lem: finite-moment} and Theorem \ref{thm: expansion-at-inv-meas-mean}).}

\mg{The next result, which is a direct consequence of Theorems 2.6 and 2.9 in~\cite{zhang2021robust}, summarizes the known iteration complexity results for SAPD with \sa{constant primal and 
dual stepsizes.}}
\begin{theorem}[\cite{zhang2021robust}]
\label{thm:complexity}
\sa{Consider the problem in~\eqref{def: sapd-opt-prob}. Under Assumptions~\ref{assump: obj-fun-prop} and~\ref{assump: noise-properties} for $p=2$, given any $\epsilon>0$, there exists $\hat{\theta}_\epsilon\in(0,1)$ such that the SAPD iterate sequence $\{x_k,y_k\}$ generated using
 $\tau=\frac{1-\theta}{\mu_x\theta}$ and $\sigma=\frac{1-\theta}{\mu_y\theta}$
for any $\theta\in \sa{[}\hat{\theta}_\epsilon, 1)$ satisfies 
{\small
$$\mathbb{E}\left[\max\{\sup_{x,y}\{f(\bar{x}_N,y)-f(x,\bar{y}_N)\},~\norm{z_N-z^*}^2\}\right]\leq \epsilon,$$} 
for all $N\geq N_\epsilon\in\bcred{\mathbb{N}_+}$ such that
{\scriptsize
\begin{align*}
    N_\epsilon=\tilde\cO\left(\Big(\frac{L_{xx}}{\mu_x}+\frac{L_{yx}}{\sqrt{\mu_x\mu_y}}+\frac{L_{yy}}{\mu_y}\Big)\log\Big(\frac{1}{\epsilon}\Big)+\Big(\frac{1}{\mu_x}+\frac{1}{\mu_y}\Big)\frac{\delta_{(2)}^2}{\epsilon}\right).
\end{align*}}}%
\end{theorem}
\sa{Here $\tilde\cO(\cdot)$ notation hides $\log(1/\epsilon)$ factor for the variance term, which can be removed adopting a restart approach proposed in~\cite{fallah_2020}. Let $$\kappa_x\triangleq L_{xx}/\mu_x, \quad \kappa_y\triangleq L_{yy}/\mu_y, \quad \kappa_{yx} \triangleq L_{yx}/\sqrt{\mu_x\mu_y},$$ and define $\kappa\triangleq \max\{\kappa_x,\kappa_y,\kappa_{yx}\}$. In the above result, the first part $\cO(\kappa\log{1/\epsilon})$ is related to the bias term; it is essential to emphasize that the bias term for SGDA using constant step size is $\cO(\kappa^2\log{1/\epsilon})$. Thus, SAPD 
accelerates the convergence of the bias term compared to the popular alternative SGDA.}

\paragraph{Dynamical system representation of SAPD.} 
\sa{Through defining the following concatenations,}
\begin{align}
z_k\triangleq\begin{bmatrix} 
x_k\\
y_k
\end{bmatrix},\;
\xi_k\triangleq \begin{bmatrix} 
z_k\\
z_{k-1}\\
\end{bmatrix}, 
\tPhi_k\triangleq\begin{bmatrix} \tn_x f(x_{k}, 
\sa{y_{k+1}})\\
\tn_y f(x_k,y_k)\\
\tn_y f(x_{k-1},y_{k-1})
\end{bmatrix},
\end{align}
\sa{the SAPD recursion in \eqref{alg: sapd}} can be rewritten as 
\begin{equation}\label{alg: sapd-dyn-sys}
\xi_{k+1}= M \xi_k + N \tPhi_k,
\end{equation}
with the associated matrices 
{\small
\begin{align*}
M=\begin{bmatrix} 
I_d & 0_d \\
I_d & 0_d
\end{bmatrix},\  
N= \begin{bmatrix}
-\tau I_{d_x} & 0_{d_x\times d_y} & 0_{d_x\times d_y}\\
0_{d_y\times d_x} &  \sigma(1+\theta)I_{d_y} & -\theta \sigma I_{d_y}\\ 
0_{d_x\times d_x} &  0_{d_x\times d_y} & 0_{d_x\times d_y}\\
0_{d_y} & 0_{d_y} & 0_{d_y} 
\end{bmatrix},
\end{align*}}%
with $d\triangleq d_x+d_y$. In some places, we will also use the notation $\xitdt_k$ interchangeably with $\xi_k$ to emphasize the iterates' dependency on the parameters of the algorithm. 

\paragraph{SAPD iterates as a Markov Chain.} We notice that under Assumption~\ref{assump: noise-properties}, the iterates $\{\xitdt_k\}_{k\in\mathbb{N}}$ define a \sa{time-homogeneous} Markov chain. Let $\mcRtdt$ be the Markov kernel of iterates $\{ \xitdt_k\}_{k\in\mathbb{N}}$ on $(\mathbb{R}^{2d},\mathcal{B}(\mathbb{R}^{2d}))$, \mg{where $\mathcal{B}(\mathbb{R}^{2d})$ is the sigma-algebra generated by the Borel sets in $\mathbb{R}^{2d}$}, i.e., 
\sa{for all $A\in\mathcal{B}(\mathbb{R}^{2d})$ and $k\geq 0$,} 
$$\mcRtdt(\xi_k,A)=\mathbb{P}\{ \xi_{k+1}\in A | \xi_k \},\quad\forall \xi_k\in\mathbb{R}^{2d},$$ 
almost surely, \mg{the map} \mg{$\xi\mapsto \mcRtdt(\xi,A)$} is Borel measurable, and $\mcRtdt(\xi_0, \cdot)$ is a probability measure on $(\mathbb{R}^{2d},\mathcal{B}(\mathbb{R}^{2d}))$ \sa{for any $\xi_0\in\mathbb{R}^{2d}$}. \sa{Then, for all $k\geq 1$, the Markov kernel associated with} $\mcRtdt_k$ of $\xi_k$ is recursively defined for any $\xi_0\in\mathbb{R}^{2d}$ and $A\in \mathcal{B}(\mathbb{R}^{2d})$ as 
$$
\mcRtdt_{k+1}(\xi_0, A) = \int_{\mathbb{R}^{\sa{2}d}}\mcRtdt_{k}(\xi_0, d\xi)\mcRtdt(\xi,A),
$$
where $\mcRtdt_1=\mcRtdt$. We also define 
$\lambda \mcRtdt_k$ for any probability measure $\lambda$ on $(\mathbb{R}^{2d}, \mathcal{B}(\mathbb{R}^{2d}))$ by 
$$
\lambda \mcRtdt_k(A) \triangleq\int_{\mathbb{R}^{\sa{2}d}}\lambda(d\xi)\mcRtdt_k(\xi,A),\  \forall~A \in \mathcal{B}(\mathbb{R}^{2d}).
$$
The above definition implies that for \mg{any} probability measure $\lambda$ on $\mathcal{B}(\mathbb{R}^{2d})$ and $k\in \mathbb{N}_{+}$, $\lambda \mcRtdt_k$  is the distribution of $\xitdt_k$ 
\sa{initialized} from \sa{$\xi_0\sim\lambda$, i.e., $\xi_0$ 
drawn from $\lambda$}. For any function $\Psi: \mathbb{R}^{2d}\rightarrow \mathbb{R}$ and $k\in\mathbb{N}_{+}$, the measurable function $\mcRtdt_k\Psi:\mathbb{R}^{2d}\rightarrow \mathbb{R}$
is defined as 
$$
\mcRtdt_k\Psi (\xi_0)\triangleq\int_{\mathbb{R}^{\sa{2d}}}\Psi(\xi)\mcRtdt_k(\xi_0,d\xi),\quad\sa{\forall~\xi_0\in\mathbb{R}^{2d}}.
$$
For any measure $\lambda$ on $(\mathbb{R}^{2d},\mathcal{B}(\mathbb{R}^{2d}))$ and any measurable function $h:\mathbb{R}^{2d}\rightarrow \mathbb{R}$, we denote $\int_{\mathbb{R}^{\sa{2d}}}h(\xi)d\lambda(\xi)$ by $\lambda(h)$ (whenever exists) for simplicity. Notice that this notation implies the equality $\lambda(\mcRtdt_k h)= (\lambda \mcRtdt_k)(h)$ for any measure $\lambda$ on $\mathcal{B}(\mathbb{R}^{2d})$ and any measurable function $h:\mathbb{R}^{2d}\rightarrow \mathbb{R}$ and $k\in\mathbb{N}_{+}$. A distribution $\ptdt$ is called an \textit{invariant measure} for $\mcRtdt$ if \bcred{$\ptdt\mcRtdt=\ptdt$}and we say the Markov chain $\{ \xitdt_k\}_{k\in\mathbb{N}}$ is \textit{stationary} if it admits an invariant measure $\ptdt_*$, \sa{i.e., $\xitdt_0$} is distributed according to $\ptdt_*$ implies that the distribution of $\xitdt_k$ is also $\ptdt_*$ for all $k\in\mathbb{N}$. \bc{Depending on the context, we use the notation $\pi_{*}$ interchangeably with $\pi_{*}^{(\tau,\delta,\theta)}$ for simplicity. Similarly, \mg{when all the parameters of the SAPD algorithm are parametrized as function of $\theta$}, then we 
\sa{adopt the notation} $\pi_*^{(\theta)} \triangleq\pi_*^{(\tau_{\theta},\sigma_{\theta},\theta)}$,} and \mg{similarly, 
$\xi_k^{(\theta)}\triangleq \xi_k^{(\tau_{\theta},\sigma_{\theta},\theta)}$ and $\mcR^{(\theta)}\triangleq\mcR^{(\tau_\theta,\sigma_\theta,\theta)}$.
}

\section{Variance Reduced SAPD}\label{sec: var-red-sapd}
\mg{
\sa{Since $f$ is SCSC, the SP problem in} \eqref{def: sapd-opt-prob} admits a unique saddle point $z_* = (x_*, y_*)$. In the following result, \sa{we obtain convergence guarantees for 
the SAPD iterates $(x_k, y_k)$ to the saddle-point 
in $L^2$.}
\sa{Recently,} convergence results for very general choice of SAPD parameters 
\sa{are} given in \cite{zhang2021robust}. Our result considers a specific choice of parameters, \sa{for which} we can prove the existence of an invariant distribution for the iterates.}  
\sa{For this specific choice of $\tau,\sigma$ in $\theta$, 
we were able to provide 
    a sharper result for the expected distance square metric in terms of constants,} compared to those available for SAPD in \cite{zhang2021robust}. 

\begin{proposition}\label{prop: conv-of-SAPD}
Let the parameters $(\tau,\delta,\theta)$ of SAPD algorithm \eqref{alg: sapd} be chosen as, 
\begin{align}
\label{eq:SAPD-parameters}
    \tau_\theta\sa{\triangleq}\frac{1-\theta}{\mu_x \theta},\quad \sigma_\theta\sa{\triangleq}\frac{1-\theta}{\mu_y \theta},\; \sa{\theta \in [\hat{\theta}, 1),}
\end{align}
where $\hat{\theta}\triangleq\max\{\hat{\theta}_1,\hat{\theta}_2\} \in (0,1)$ and
\begin{align*}
    \hat\theta_1\triangleq \sa{\frac{1}{1+(\kappa_x+4\kappa_{yx}^2)^{-1}}<1}, \quad\quad\quad
    \hat\theta_2\triangleq\sa{\frac{2}{{\sqrt{(1+\tfrac{1}{8\kappa_y^2})^2+\frac{1}{2\kappa_y^2}}+1+\tfrac{1}{8\kappa_y^2}}}<1},
\end{align*}
where \sa{$\kappa_x=L_{xx}/\mu_x$, $\kappa_y=L_{yy}/\mu_y$ and $\kappa_{yx}=L_{yx}/\sqrt{\mu_x\mu_y}$,} the iterates $\{z_k\}_{k\in \mathbb{N}_{+}}$ satisfies, 
\begin{align} 
\mathbb{E}\Big[{\mu_x}\Vert x_k-x_*\Vert^2 + \frac{\mu_y(1+\theta)}{2}\Vert y_k-y_*\Vert^2\Big]  \leq \theta^{k} \bcred{\Delta_0} + (1-\theta^k)\tilde{\Xi}_\theta \delta_{(2)}^2, \label{cond: SAPD-convergence-result}
\end{align}
where $z_*^\top=[x_*^\top,y_*^\top]$, $\Delta_0\triangleq{\mu_x}\Vert x_0-x_*\Vert^2 + {\mu_y}\Vert y_0 - y_*\Vert^2$, 
\begin{align*} 
\tilde\Xi_{\theta}\triangleq \frac{1-\theta}{\theta}\left(\frac{2}{\mu_x}+\frac{4}{\mu_y}\Big((1+\theta)^2+\theta^2\Big)\right).
\end{align*}

\end{proposition}
\mg{We are going to show that the invariant measure exists for the given choice of parameters in Proposition \ref{prop: conv-of-SAPD}. One 
\sa{complication is that the Markov Chain corresponding to SAPD iterates becomes non-reversible due to the momentum term;} this can also be seen from the fact that the matrices $M$ and $N$ in the representation \eqref{alg: sapd-dyn-sys} are non-symmetric. Consequently, standard tools for showing the existence of a stationary measure for reversible Markov chains are not applicable to our setting. However, the results 
in~\cite{hairer2011yet} do not require reversibility and are applicable to our setting. More specifically, \cite{hairer2011yet} shows the existence of a stationary  distribution for a Markov Chain provided that \sa{the} so-called ``minorization and drift" conditions hold with certain parameters (see the appendix for the details). In the following, we introduce an assumption on the gradient noise which says that the gradient noise admits a continuous density. This assumption allows us to show that the ``minorization condition" holds.}

\begin{assumption}\label{assump-cont-density} The gradient noise admits a continuous density, 
\sa{i.e., $p$ defined such that $p(\xi, z) dz\triangleq\mathbb{P}( z_{k+1} \in dz | \xi_{k}=\xi)$ is continuous in $(\xi,z)$.} 
\end{assumption}
\mg{The \bcred{minorization} condition from \cite{hairer2011yet} for the transition kernel $\mathcal{R}$ of SAPD iterations \sa{shown in~\eqref{alg: sapd}} requires the existence of a Lyapunov function $\mathcal{V}:\mathbb{R}^{2d}\rightarrow \bcred{[0,\infty)}$ and constants $K\geq 0$ and $\zeta \in (0,1)$ such that 
\begin{equation}
(\bcred{\mcRtdt}\mathcal{V})(\sa{\xi}) \leq \zeta \mathcal{V}(\sa{\xi}) + K,
\label{cond-major-main-text}
\end{equation}
for every $\sa{\xi} \in \mathbb{R}^{2d}$. For this purpose, we devised the following Lyapunov function 
}
\begin{multline}\label{def: lyap-func}
    \mathcal{V}(\xi_k)\sa{\triangleq}\frac{\theta}{1-\theta}\Big(\frac{\mu_x}{4}(\Vert x_k-x_*\Vert^2+\Vert x_{k-1}-x_*\Vert^2)+ \frac{\mu_y}{\sa{8}}(1+\theta)(\Vert y_k-y_*\Vert^2 + \Vert y_{k-1}-y_*\Vert^2)\Big),
\end{multline}
\mg{for which we can show that \eqref{cond-major-main-text} holds 
\sa{for some $\zeta\in(0,1)$ and $K\geq 0$} (see the appendix for more details). Then, building on the results of \cite{hairer2011yet}, we 
\sa{\mg{establish in the following result} that the distributions of SAPD iterates converge to the invariant distribution, and also give a rate result for this convergence} provided that the variance of the gradient noise is not too large. \mg{The proof is deferred to the appendix}.}%
\begin{theorem}\label{thm: inv-meas-exists}  \mg{Consider the SAPD algorithm with parameters given by \eqref{eq:SAPD-parameters}. Suppose that Assumptions \ref{assump: obj-fun-prop}, \ref{assump: noise-properties} and \ref{assump-cont-density} hold, and} 
the variance \bcred{bound} 
\sa{$\delta_{(2)}^2$} of \mg{the noise is small} such that 
\begin{equation}
\sa{{(1+\theta)}\left(\frac{4}{\mu_x}+\frac{8}{\mu_y}\Big((1+\theta)^2+\theta^2\Big)\right)\delta_{(2)}^2}\leq R,
\end{equation}
where $R$ is a constant specified in the Appendix. Then \sa{for any initialization $\xi_0 \in \mathbb{R}^{2d}$,} SAPD iterates 
admit a 
unique invariant measure 
$\pi^{(\theta)}_*$, \sa{i.e., 
$\lambda_{\xi_0}\mcR_k^{(\theta)}$, the distribution of \bcred{$\xi_k^{(\theta)}$}, converges to $\sa{\pi^{(\theta)}_*}$, where $\lambda_{\xi_0}$ denotes the Dirac distribution at $\xi_0$.} Moreover, there exists \sa{$C>0$} 
such that 
$$
\Vert \sa{\mcR_k^{(\theta)}} \psi - \sa{\pi^{(\theta)}_*}(\psi)\Vert \leq \sa{C} 
\sa{\left(\frac{2\theta}{1+\theta}\right)^k}\Vert \psi - \sa{\pi^{(\theta)}_*}(\psi)\Vert,
$$
for \bc{any initialization $\xi_0$} and every measurable function $\psi$ such that $\Vert \psi \Vert <\infty$, where $\Vert \psi \Vert \triangleq \sup_{\xi} \frac{|\psi(\xi)|}{1+\mathcal{V}(\xi)}$ is the weighted supremum norm.
\end{theorem}
\mg{\sa{Since $\theta\in(\hat\theta,1)$,} by taking limit superior of both sides in \eqref{cond: SAPD-convergence-result} as $k\to \infty$, it can be seen that we have $$\limsup_{k\to\infty} \mathbb{E}[\|z_k - z_*\|^2] = \mathcal{O}(1-\theta)~\sa{\delta^2_{(2)}}.$$
Since the distribution of the SAPD iterates $z_k$ converges to the stationary distribution based on Theorem \ref{thm: inv-meas-exists}, we can replace the limit superieure with a limit to obtain $$\sa{\mathbb{E}[\|z_\infty - z_*\|^2] =} \lim_{k\to\infty}\mathbb{E}[\|z_k - z_*\|^2] =  \mathcal{O}(1-\theta)~\sa{\delta^2_{(2)}},$$
with the convention that $z_\infty$ is distributed according to the stationary distribution. This is summarized in Lemma \ref{lem: finite-moment} below, where we also characterize the fourth moments of the stationary distribution as a function of $\theta$ as $\theta \to 1$. In particular, obtaining the fourth moments require a careful non-trivial Lyapunov analysis of the SAPD algorithm that involves higher-order interactions between the gradient noise and the iterates; the proof is deferred to the appendix.
}

\begin{lemma}\label{lem: finite-moment}
\mg{
\sa{Under the premise} of Theorem \ref{thm: inv-meas-exists}, 
$\mE[\Vert z_\infty-z_*\Vert^2]=\mathcal{O}(1-\theta)$ and $\mE[\Vert z_\infty-z_*\Vert^4]=\mathcal{O}\left((1-\theta\right)^2)$,
where $z_\infty$ is a random variable distributed according to the stationary distribution of the SAPD iterates.}
\end{lemma}
\mg{Building on this lemma, in the next result, we show that the expected iterates contain a bias that is 
\sa{$\cO(1-\theta)$} as $\theta \to 1$. The bias stems mainly from the fact that we use constant primal and dual stepsizes in SAPD as opposed to decaying stepsizes.} \bc{Our proof is based on relating the bias to the equilibrium covariance of the iterates and characterizing the equilibrium covariance as the solution to a set of coupled algebraic Lyapunov equations (whose solutions can then be approximated as $\theta\to 1$).}
\begin{theorem}\label{thm: expansion-at-inv-meas-mean}
\sa{Under the premise} of Theorem \ref{thm: inv-meas-exists}, \sa{the limit of expected SAPD iterate sequence, i.e.,}
$$\bar{\xi}^{(\theta)}\triangleq\lim_{k\rightarrow \infty}\mathbb{E}[\xi_k^{(\theta)}],$$
\sa{exists} at stationarity 
\sa{such that}
\begin{equation}\label{eq: inv-meas-mean}
    \bar{\xi}^{(\theta)}-\xi_*=(1-\theta)\bcred{\begin{bmatrix} (\hes)^{-1} \left( \nabla^{(3)}f_*M\right)\\
    (\hes)^{-1} \left( \nabla^{(3)}f_*M\right)
    \end{bmatrix}}+\mathcal{O}\left((1-\theta)^{3/2}\right),
\end{equation}
as $\theta \to 1$
where \bcred{$\xi_{*}=(z_*,z_*)$,} $M$ is a fixed matrix \mg{with an explicit formula provided in the appendix}, \mg{$\hes$ is the Hessian matrix at the saddle point $z_*$, and $\nabla^{(3)}_*f$ is the tensor of order 3 that contains the third-order partial derivatives of $f$ at $z_*$.}
\end{theorem}
\mg{Theorem \ref{thm: expansion-at-inv-meas-mean} shows that the expected iterates admit a bias that is $\mathcal{O}(1-\theta)$}. \mg{Following a similar approach to \sa{Richardson-Romberg} extrapolation techniques 
\sa{adopted} for stochastic gradient descent methods \sa{for} unconstrained optimization in~\cite{bridging-the-gap}, 
\sa{our result in \eqref{eq: inv-meas-mean}} suggests that if we generate two SAPD sequence $\{\xi_k^{(\theta_1)}\}$ and $\{\xi_k^{(\theta_2)}\}$, then 
\sa{through appropriately forming a weighted sequence using the two, one can eliminate a significant portion of the bias.} 
\sa{Indeed,} if we choose the parameters as $\theta_2=2\theta_1-1$, then the from \bcred{Theorem \ref{thm: expansion-at-inv-meas-mean}}, we obtain 
$$
2\bar{\xi}^{(\theta_1)}-\bar{\xi}^{(\theta_2)}= \xi_*+\mathcal{O}\left((1-\theta)^{3/2}\right),
$$}%
\sa{i.e., the bias improves from $\cO(1-\theta)$ to $\cO\left((1-\theta)^{3/2}\right)$.} \mg{That being said, the variance of $\{2\xi_k^{(\theta_1)}-\xi_k^{(\theta_2)}\}_k$ sequence could be larger than that of $\{\xi_k^{(\theta_1)}\}$. For controlling the variance, a standard idea in stochastic approximation methods is to consider the average of the iterates, \sa{i.e., consider $\{\tilde{\xi}_k^{(\theta)}\}_{k\in\mathbb{N}_+}$ such that for $k\geq 1$,} }
\begin{equation}
\bcred{\tilde{\xi}_k^{(\theta)}\sa{\triangleq} \frac{1}{k}\sum_{i=0}^{k-1}\xi_{i}^{(\theta)}}.
\label{def-aver-seq}
\end{equation}
\mg{Often, based on central limit theorem-type arguments, 
$\tilde{\xi}_k^{(\theta)}$ improves the variance of $\xi_k^{(\theta)}$ due to its \sa{averaged} nature. 

Next, in the following result, we study the gap between the mean of the averaged 
\sa{iterate} $\tilde{\xi}_k^{(\theta)}$ and \bcred{$\bar{\xi}^{(\theta)}$} defined in Theorem \ref{thm: expansion-at-inv-meas-mean}}. The result shows that this gap decays polynomially fast in $k$ and is of the order $\mathcal{O}(1/k)$.

\begin{theorem}\label{thm: av-seq-conv}
\mg{
\sa{Under the premise} of Theorem \ref{thm: inv-meas-exists}, let $\xi_0$ be the initialization of the SAPD iterates
 $\{\xi_k^{(\theta)}\}_{k\in\mathbb{N}}$. 
\sa{For any $\xi_0$ and \bcred{$k\geq 1$}}, 
\begin{equation}\label{eq: gap-avg-mean-and-inv-meas-mean}
 \mathbb{E}[\tilde{\xi}_k^{(\theta)}]=\bar{\xi}^{(\theta)}+\frac{1+\theta}{1-\theta} \left(1-\left(\frac{2\theta}{1+\theta}\right)^k\right)\frac{C}{k}~\norm{\bcred{\xi_0-\bar{\xi}^{\theta}}},
\end{equation}
where \sa{$C>0$ is the constant defined in Theorem~\ref{thm: inv-meas-exists}} and $\tilde{\xi}_k^{(\theta)}$ is the averaged 
\sa{iterate} defined in \eqref{def-aver-seq}}.
\end{theorem}
From the equations \eqref{eq: inv-meas-mean} and \eqref{eq: gap-avg-mean-and-inv-meas-mean}, we obtain that  
\begin{equation}\label{eq: mean-av-seq}
\mathbb{E}[\tilde{\xi}_k^{(\theta)}]- \xi_* = 
\sa{\frac{1+\theta}{1-\theta}~\frac{C}{k}\norm{\xi_0-\bar{\xi}^{\theta}}}+ (1-\theta) \bcred{\begin{bmatrix}(\hes)^{-1} \nabla^{(3)}f_*M\\(\hes)^{-1} \nabla^{(3)}f_*M \end{bmatrix}} + r_{\theta},
\end{equation}
where $\Vert r_{\theta}\Vert <C_3(1-\theta)^{3/2} + \sa{\tfrac{1}{k}} e^{-k C_4(1-\theta)}$ for some constants $C_3, C_4>0$. In particular, given that the error between $\mathbb{E}[\tilde{\xi}_k^{(\theta)}]$ and $\bar{\xi}^{\theta}$ is vanishing, 
we could apply the same variance reduction ideas to the averaged sequence \bcred{$\{\tilde{\xi}_k^{(\theta)}\}$} to reduce the variance. In the following section, we provide numerical examples to illustrate the benefits of our variance reduced SAPD method.

\section{Numerical Experiments}
In this section, we compare the performance of VR-SAPD with SAPD \cite{zhang2021robust}, S-OGDA \cite{fallah_2020}, and SMP \cite{SMP} on the $\ell_2$-regularized distributionally robust learning problem
\begin{equation}\label{def: DRO}
    \min_{x\in \mathcal{S}}\max_{y\in \mathcal{P}_r}\sa{f_{0}}(x,y)\triangleq \frac{\mu_x}{2}\Vert x \Vert^2 + \sum_{i=1}^{n}y_i \phi_i(x)\tag{DRO},
\end{equation}
where $\mathcal{S}\triangleq \{x \in\mathbb{R}^d : \Vert x \Vert^2 \leq \mathcal{D}_x \} $ for some diameter $\mathcal{D}_x>0$ and $\phi_i:\mathbb{R}^d\rightarrow \mathbb{R}$ \sa{is a convex loss function corresponding to the $i$-th data point for $i=1,\ldots,n$.} 
\mg{While \sa{for} the 
\sa{classic} empirical risk minimization \sa{problem} 
the weights are selected 
\sa{as} $y_i = 1/n$ for every $i=1,\ldots,n$, the robust formulation \sa{in}~\eqref{def: DRO} \sa{allows for $\{y_i\}_{i=1}^n$ to 
come from} an uncertainty set around the uniform weights
$\mathcal{P}_r \triangleq \{ y \in \mathbb{R}_+^n : \mathbf{1}^\top y=1, \Vert y-\mathbf{1}/n \Vert^2 \leq \frac{r}{n^2} \}$  whose radius is controlled by the parameter $r$, where $\mathbf{1}$ is the vector 
\sa{of all ones} --see the experimental section in~\cite{zhang2021robust} for more details.}

\mg{The DRO formulation is 
\sa{affine} in the variable $y$; \sa{hence,} not strongly concave in $y$. However, in a similar spirit to Nesterov's smoothing technique \cite{Nesterov}, we can introduce a 
dual regularizer to have a SCSC approximation to 
\eqref{def: DRO}}. \sa{This corresponds to approximating the original problem in the primal form that is non-smooth with a smooth primal optimization problem. Indeed, we will implement VR-SAPD algorithm on}
\begin{small}
\begin{equation} \label{def: SCSC-problem}
\min_{x\in \mathcal{S}} \max_{y\in \mathcal{P}_r} f(x,y) \triangleq \frac{\mu_x}{2}\Vert x \Vert^2 + \sum_{i=1}^{n}y_i \phi_i(x) - \frac{\mu_y}{2}\Vert y\Vert^2 \tag{\sa{Reg-DRO}},
\end{equation}
\end{small}%
where $\mu_y$ is chosen appropriately, \sa{i.e., $\mu_y=\Theta(\epsilon)$.} \bc{In particular, if we choose  $\mu_y=\frac{\epsilon}{2\mathcal{D}_y}$ where $\mathcal{D}_y\sa{\triangleq}\sup_{y\in \mathcal{P}_r}\Vert y\Vert^2=1$ \sa{for} \eqref{def: SCSC-problem} and find a solution $(\bar{x},\bar{y})$ satisfying $\mE[\sup_{(x,y)\in\mathcal{S}\times \mathcal{P}_r}\sa{\{}f(\bar{x},y)-f(x,\bar{y})\sa{\}}]<\frac{3\epsilon}{4}$, then it can be shown that $(\bar{x},\bar{y})$ also satisfies $\mE[\sup_{(x,y)\in\mathcal{S}\times \mathcal{P}_r}\sa{\{f_0}(\bar{x},y)-\sa{f_0}(x,\bar{y})\sa{\}}]<\epsilon$.}  
In the following, we consider the binary logistic loss function $\phi_i(x)=\log(1+\exp(-b_i a_i^\top x))$ calculated at \sa{the $i$-th} data point from $\{a_i,b_i\}_{i\in \{1,..,n\}}$, and we set $r=2\sqrt{n}$. Notice that the objective of \eqref{def: SCSC-problem} admits Lipschitz constants $L_{xy}=L_{yx}=\Vert A \Vert_2$, $L_{xx}=\max_{i=1,..,n}\{\frac{1}{4}\Vert a_i \Vert^2_2 \}$, and $L_{yy}=0$ where $A\in\mathbb{R}^{n\times d}$ is the data matrix with rows $\{a_i\}_{i=1}^{n}$ and columns $\{A_j\}_{j=1}^d$. 
\sa{To stay feasible at each iteration, after performing $y_{k+1}$ and $x_{k+1}$ steps as in \eqref{alg: sapd}, we computed the \sa{Euclidean} projection onto} the constraint sets $\mathcal{S}$ and $\mathcal{P}_r$. \mg{In our theoretical results, for the simplicity of the presentation, we considered SAPD without projection steps. However, our convergence results to the stationary distribution extend in a straightforward manner when such projection steps are included.}\footnote{\mg{This is because the projection is a non-expansive operation \cite{bertsekas2015convex} which does not increase the distance to the solution and therefore the minorization and drift conditions will still hold for our distance-based Lyapunov function \eqref{def: lyap-func} after minor modifications to the proof of Theorem \ref{thm: inv-meas-exists}}.} 


To illustrate our results from Section \ref{sec: var-red-sapd}, 
we consider two sequences $\xi_k^{(\theta_1)}$ and $\xi_{k}^{(\theta_2)}$, \sa{both initialized with the same point $\xi_0$,} where 
\sa{such that} $\theta_2=2\theta_1-1$; and then consider the performance of the interpolated sequence $2\tilde{\xi}_{k}^{(\theta_1)}-\tilde{\xi}_{k}^{(\theta_2)}$. 
\begin{figure}[h]
    \centering
    \includegraphics[width=0.7\linewidth]{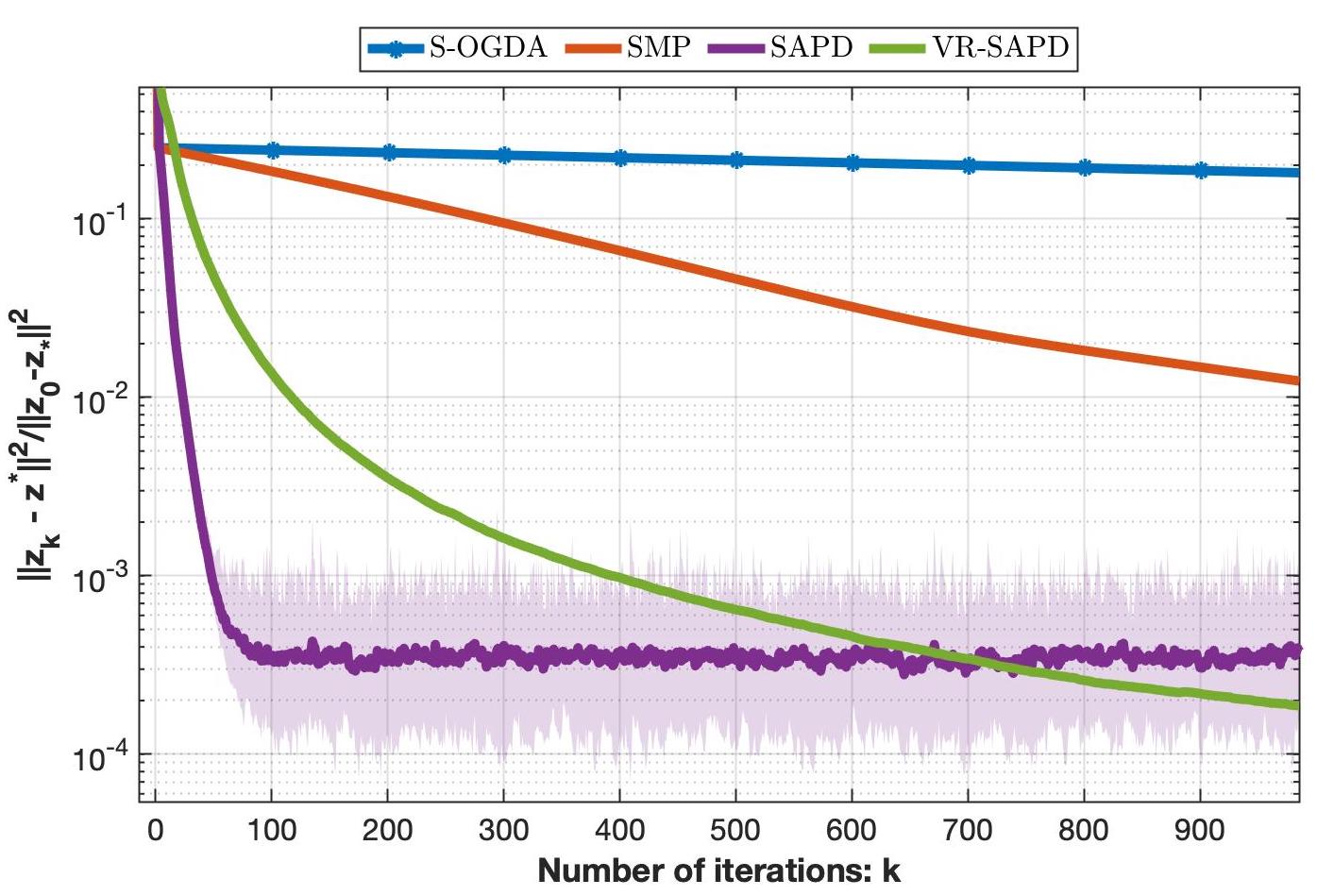}
    \caption{Comparison of S-OGDA, SMP, SAPD, and VR-SAPD on the Arcene data set in terms of the \bcred{relative} expected distance squared \bcred{$\mathbb{E}\|z_k - z_*\|^2/\Vert z_0-z_*\Vert^2$} as a performance metric.}
    \label{fig: arcene-fig}
\end{figure}
\begin{figure}[ht]
    \centering
    \includegraphics[width=0.7\linewidth]{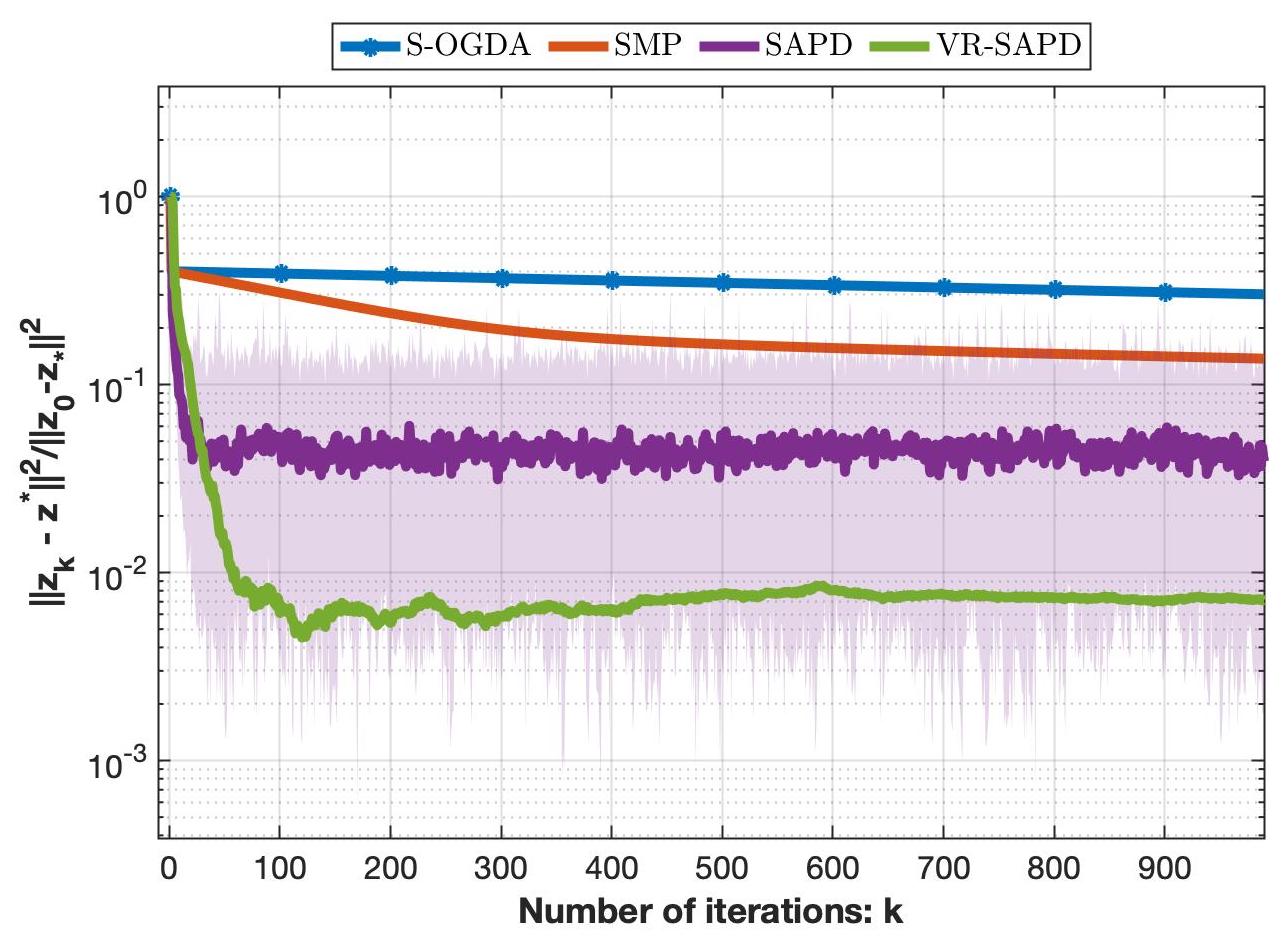}
    \caption{Comparison of S-OGDA, SMP, SAPD, and VR-SAPD on the Dry Bean data set in terms of the relative expected distance squared $\mathbb{E}\|z_k - z_*\|^2/\Vert z_0-z_*\Vert^2$ as a performance metric.}
    \label{fig: drybean-fig}
\end{figure}
\begin{figure}[h]
    \centering
    \includegraphics[width=0.7\linewidth]{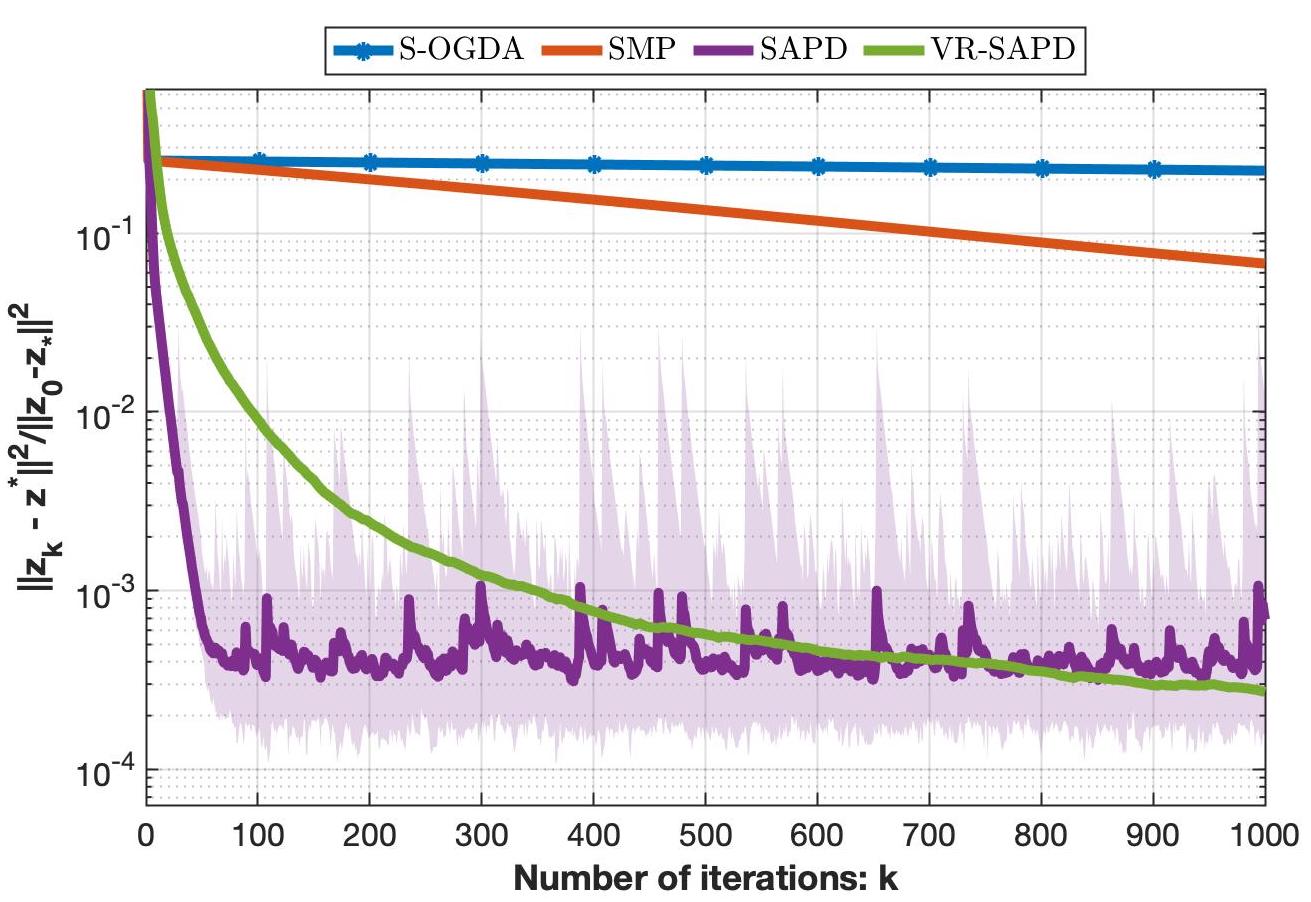}
    \caption{Comparison of S-OGDA, SMP, SAPD, and VR-SAPD on the MNIST data set in terms of the \bcred{relative} expected distance squared \bcred{$\mathbb{E}\|z_k - z_*\|^2/\Vert z_0-z_*\Vert^2$} as a performance metric.}
    \label{fig: mnist-fig}
\end{figure}%

We consider the algorithms VR-SAPD 
, SAPD, S-OGDA, and SMP on three datasets:
\begin{enumerate}
    \item Dry Bean dataset \cite{DryBean} with $(n,d)=(9528, 16)$;
    \item Arcene dataset \cite{Arcene} with $(n,d)=(97,10000)$;
    \item MNIST dataset \cite{lecun1998gradient2} for hand-written digit classification with \bcred{$(n,d)=(1000,400)$}.
\end{enumerate}
Throughout the experiments, we set the regularization parameter $\mu_x$ through cross-validation and defined the parameters $(\tau_{\theta}, \sigma_\theta)$ of both VR-SAPD and SAPD as given in Proposition \ref{prop: conv-of-SAPD}, tuning the parameter as $\theta=0.95$. The gradients are estimated from mini-batches of data sampled at random with replacement where we have a batch size of $10$. At Dry Beans, we set $\mu_x=0.01$, $\mu_y=10$. 
We normalize each feature column $A_j\in\mathbb{R}^n$ using $A_j= \frac{A_j-\min(A_j)}{\max(A_j)-\min(A_j)}$ where max and min are taken over the elements of $A_j$. For the Arcene data set, we set $\mu_x=0.02$ and $\mu_y=10$ and normalize the data as $A\leftarrow A/\min\{\sqrt{d},\sqrt{n}\}$. Lastly, we set $\mu_x=0.1$, $\mu_y=10$ and mini-batch size $10$ for MNIST. We set $b_i=1$ if the i-th data point is number four and set $b_i=-1$, otherwise.

We generated two sequences $\{z^{(\theta_1)}_k\}_{k\in \mathbb{N}}$ and $\{z^{(\theta_2)}_k\}_{k\in\mathbb{N}}$ with $\theta_1=\theta$ and $\theta_2=2\theta-1$. For VR-SAPD, we considered the extrapolated sequence $z_k=2z_k^{(\theta_1)}-z_k^{(\theta_2)}$. The step-size of S-OGDA and SMP are taken as $\frac{1}{L}$ and $\frac{1}{\sqrt{L}}$, respectively, where $L=\max\{L_{xx}+\mu_x, \mu_y, L_{xy}, L_{yx}\}$.
\bcred{We generated 50 random sample paths to compute the mean and the variance of $\{\Vert z_k-z_*\Vert^2\}_{k=1}^K$ generated by the algorithms, \sa{each running for $K=1000$ iterations},}
where we 
\sa{closely approximate} $z_*$ by running the deterministic APD algorithm~\cite{hamedani2021primal} until the change the loss function is below $\%0.01$. \mg{Our results for each dataset are summarized in Figures \ref{fig: arcene-fig}, \ref{fig: drybean-fig} and \ref{fig: mnist-fig} where we compared the methods in terms of the expected \bcred{relative} \sa{change in the} distance squared, \sa{i.e.,} $\mathbb{E}[\bcred{\frac{\|z_k - z_*\|^2}{\Vert z_0-z_*\Vert^2}}]$, as a performance metric \sa{--all the algorithms are initialized from 
$x_0= 2~\mathbf{1}_{d_x}$ and $y_0=1/d_{y}\mathbf{1}_{d_y}$.} 
For each method, we also display the standard deviation \sa{as the shaded region around the average of 50 sample paths.} Our results show that our variance reduction techniques result in a smaller asymptotic variance for the SAPD iterates as intended. We also see that VR-SAPD iterates admit a smaller variance in general compared to SAPD and improves upon the other methods SMP and S-OGDA. These results showcase the benefits of our variance reduction techniques.}

\section{Conclusion}
\mg{Under some assumptions, we showed that with constant stepsize, the distribution of the SAPD iterates admit a stationary distribution and characterized its dependency to the SAPD parameters. Based on these results, we introduced a variance-reduced stochastic accelerated primal-dual method (VR-SAPD), 
\sa{which is} based on Richardson-Romberg extrapolation. We also illustrated the efficiency and benefits of \sa{our 
accelerated variance-reduced method} on robust training of logistic regression models.} 

\section{Acknowledgements}
This work was funded by the grants ONR N00014-21-1-2244, \sa{ONR N00014-21-1-2271}, NSF CCF-1814888 and NSF DMS-2053485. \bcred{We thank Xuan Zhang for providing \sa{us with the 
MATLAB implementations of SAPD, S-OGDA and SMP for our} numerical experiments.}




\bibliographystyle{unsrtnat}
\bibliography{bugra_ref.bib}  
\newpage
\onecolumn
\section{Appendix}\label{sec: appendix}
\subsection{Proof of Proposition~\ref{prop: conv-of-SAPD}}
\sa{In order to prove Proposition~\ref{prop: conv-of-SAPD}, we first provide a technical lemma about a specific SAPD parameter selection, and next we give an adaptation of a convergence result from \cite{zhang2021robust}.} 

It has been shown \sa{in Remark 2.4 of \citep{zhang2021robust} that} \sa{the SAPD parameters $\tau,\sigma>0$ and $\theta\in(0,1)$} are admissible if they satisfy the \mg{inequalities} given in \eqref{cond: MI-2} \sa{for some $\alpha>0$.} \mg{In the next result, we show that our choice of parameters satisfy these inequalities for some $\alpha>0$.}
\begin{lemma}
\label{lem:MI-cond}
\sa{The SAPD parameters $\tau,\sigma>0$ and $\theta\in(0,1)$ chosen as in \eqref{eq:SAPD-parameters} satisfy}
\begin{align}\label{cond: MI-2}
    \min\{\tau \mu_x, \sigma \mu_y \}\geq \frac{1-\theta}{\theta}, \quad \begin{bmatrix} 
    \frac{1}{\tau}-L_{xx} & 0 & -L_{yx}\\ 
    0 & \frac{1}{\sigma}-\alpha & -L_{yy}\\ 
    -L_{yx} & -L_{yy} & \frac{\alpha}{\theta}
    \end{bmatrix} \succeq 0,
\end{align}
\sa{for $\alpha=\frac{1-\theta}{2\sigma}\in(0,\frac{1}{\sigma})$.}
\end{lemma}
\begin{proof}
We show that \sa{when $\theta \in (0,1)$ sufficiently close to $1$, these inequalities hold for the specific parameter choice given in \eqref{eq:SAPD-parameters}, i.e., $\tau= \frac{1-\theta}{\mu_x \theta}$ and $\sigma=\frac{1-\theta}{\mu_y \theta}$ for any $\theta\in[\hat\theta,1)$, satisfies \eqref{cond: MI-2} when $\alpha= \frac{1-\theta}{2\sigma}$.} 
\mg{With this choice of parameters,} it is clear that $\tau$ and $\sigma$ satisfy the first inequality on the left hand-side of \eqref{cond: MI-2}. \mg{It remains to prove that the matrix inequality in \eqref{cond: MI-2} is also satisfied. If we write}
\begin{equation} 
\begin{bmatrix} 
    \frac{1}{\tau}-L_{xx} & 0 & -L_{yx}\\ 
    0 & \frac{1}{\sigma}-\alpha & -L_{yy}\\ 
    -L_{yx} & -L_{yy} & \frac{\alpha}{\theta}
    \end{bmatrix}= \underbrace{\begin{bmatrix} 
    \frac{1}{\tau}- L_{xx} & 0 & -L_{yx} \\ 
    0 & 0& 0\\
    -L_{yx} & 0 & \frac{\alpha}{2\theta}
    \end{bmatrix}}_{G_1}+ \underbrace{\begin{bmatrix} 
    0  & 0 & 0 \\    
    0 & \frac{1}{\sigma} -\alpha & -L_{yy} \\ 
    0 & -L_{yy} & \frac{\alpha}{2\theta}
    \end{bmatrix}}_{G_2},
\end{equation}
\mg{\sa{then} it suffices to show that the eigenvalues of the matrices $G_1$ and $G_2$ are non-negative.} The characteristic function of $G_1$ \mg{can be computed} as 
\begin{align*}
    \det(G_1-\lambda I_3)=-\lambda \left[ \lambda^2 -\lambda\left( \frac{1}{\tau}-L_{xx}+\frac{\alpha}{2\theta}\right)+\frac{\alpha}{2\theta}(\frac{1}{\tau}-L_{xx})-L_{yx}^2 \right],
\end{align*}
\mg{which has 3 roots, one being at $\lambda=0$.}
Notice that $\theta \mg{\geq} \hat{\theta}_1 \geq\frac{1}{1+\frac{\mu_x}{L_{xx}}} $ implies $0 \mg{\leq} \frac{\theta \mu_x}{1-\theta}-L_{xx} = \frac{1}{\tau}-L_{xx}$. Moreover, since $\alpha= \frac{1-\theta}{2\sigma}= \frac{\theta \mu_y}{2}$, we have 
$$
\frac{\alpha}{2\theta}\left( \frac{1}{\tau}-L_{xx}\right)-L_{yx}^2 = \frac{\mu_{y}}{4(1-\theta)}\left[ \theta (\mu_x+L_{xx}+\frac{4L_{yx}^2}{\mu_y})-(L_{xx}+\frac{4L_{yx}^2}{\mu_y})\right],
$$
which is \bc{non-negative} for $\theta \mg{\geq} \hat{\theta}_1$. This implies that the sum and the products of the roots of $\det(G_1-\lambda I_3)$ are positive for \mg{$\theta {\geq} \hat{\theta}_1$}, i.e we have $G_1 \succeq 0$ if $\theta {\geq}
\hat{\theta}_1$. 
 Similarly, we write the characteristic function of $G_2$, 
\begin{align*} 
\det(G_2-\lambda I_3) = -\lambda \left[ \lambda^2 - \lambda (\frac{\alpha}{2\theta}+\frac{1}{\sigma}-\alpha)+\frac{\alpha}{2\theta}(\frac{1}{\sigma}-\alpha)-L_{yy}^{2} \right],
\end{align*}
which has 3 roots, one being at $\lambda=0$. By a straightforward computation \mg{considering the trace of $G_2$}, it follows that the sum of the roots of $\det(G_2-\lambda I_3)$ is positive; therefore, the eigenvalues of $G_2$ are \mg{non-negative if} the following inequality holds 
$$
0 \mg{\leq} \frac{\alpha}{2\theta}(\frac{1}{\sigma}-\alpha)-L_{yy}^2 = \frac{\mu_y^2}{8(1-\theta)}\left[ \theta^2 + \theta \left( 1+\frac{8L_{yy}^2}{\mu_y^2}\right)-\frac{8L_{yy}^2}{\mu_y^2} \right].
$$
The polynomial on the right-side is \mg{non-negative} for $\theta \geq \hat{\theta}_2$ \mg{in which case we have $G_2 \succeq 0$}. This proves that the conditions \eqref{cond: MI-2} hold for $\tau= \frac{1-\theta}{\mu_x \theta}$, $\sigma=\frac{1-\theta}{\mu_y\theta}$ and $\alpha= \frac{1-\theta}{2\sigma}$ as long as $\theta \geq \max\{ 
\hat{\theta}_1, \hat{\theta}_2\}$. This completes the proof.
\end{proof}

\sa{Although some conclusions of [Theorem 2.1,~\cite{zhang2021robust}] requires a compactness assumption on the domain of the objective, 
it is worth emphasizing that the particular result [ineq. (2.3), \cite{zhang2021robust}] 
is still valid even when the domain is unbounded. We will use this result to prove Proposition \ref{prop: conv-of-SAPD}.} For the sake of completeness, we state 
\sa{a simplified version of} [Theorem 2.1., \cite{zhang2021robust}] below which is applicable to \eqref{def: sapd-opt-prob} where the domain is unbounded.
\begin{theorem}[{Adaptation of [Theorem~2.1,~\citep{zhang2021robust}]}]
\label{thm: theorem-7}
Suppose Assumption~\ref{assump: obj-fun-prop} and Assumption~\ref{assump: noise-properties} 
\sa{hold}. 
Let $\{x_k,y_k \}_{k\geq 0}$ \sa{be the SAPD iterate sequence generated according to \eqref{alg: sapd}}, using $\tau,\sigma>0$ and \sa{$\theta \in (0,1)$ that satisfy \eqref{cond: MI-2}}
for some \sa{$\alpha \in (0,\frac{1}{\sigma})$}. 
Then the following inequality holds for any $x_0 \in \mathbb{R}^{d_x}$ and $y_0 \in \mathbb{R}^{d_y}$, 
\begin{multline}\label{ineq: lyapunov-decay}
\mathbb{E}\left[\frac{1}{2\tau}\Vert x_N-x_*\Vert^2+ \frac{1-\alpha\sigma}{2\sigma}\Vert y_N-y_*\Vert^2\right] \leq \frac{\sa{\theta}^{N}}{2} \left( \frac{1}{\tau}\Vert x_0-x_*\Vert^2 + \frac{1}{\sigma}\Vert y_0 - y_*\Vert^2 \right) + \frac{\theta (1-\theta^N)}{1-\theta} \sa{\tilde\Xi_{\tau,\sigma,\theta}\delta_{(2)}^2},
\end{multline}
where the constant \sa{$\tilde\Xi_{\tau,\sigma,\theta}\triangleq\tau+2\sigma\big((1+\theta)^2+\theta^2\big)$.} 
\end{theorem}
\begin{proof}
\sa{The result immediately follows from [Theorem 2.1,~\cite{zhang2021robust}] and [Remark 2.4,~\cite{zhang2021robust}].
More precisely, in the proof of [Theorem 2.1,~\cite{zhang2021robust}], since $(x^*,y^*)$ is a saddle point, setting $x=x^*$ and $y=y^*$ in [(A.17), (A.28) and (A.29),~\cite{zhang2021robust}]
and adding $\rho^{-N+1}\Delta_N(x^*,y^*)$ 
to both sides implies that 
\begin{equation}\label{ineq: var-bound-sapd}
\frac{1}{2\theta \tau}\Vert x_N-x_*\Vert^2+ \frac{1-\alpha \sigma}{2\theta \sigma}\Vert y_N-y_*\Vert^2 \leq \theta^{N-1} \left(\frac{1}{2\tau}\Vert x_0-x_*\Vert^2 + \frac{1}{2\sigma}\Vert y_0-y_*\Vert^2 + \sum_{k=0}^{N-1}\theta^{-k} F_k\right),
\end{equation}
where
\begin{equation}
\label{eq:Fk}
F_k\triangleq -\fprod{\tn_xf(x_k,{y}_{k+1})-\nabla_x f(x_k,{y}_{k+1}),~ x_{k}-x_* -\tau \tn_x f(x_k,{y}_{k+1})}+\fprod{\tilde{q}_k-q_k,~y_{k+1}-y_*},
\end{equation}
with \begin{equation} {q}_k\triangleq (1+\theta)\grad_y f(x_k,y_k)-\theta \grad_y f(x_{k-1},y_{k-1}).
\label{def-qk}
\end{equation}
Using the fact that $y_{k+1}=y_k+\sigma\tilde q_k$ 
\mg{and by adding the term $\tau \tn_xf(x_k,{y}_{k+1})$ to the dot product in \eqref{eq:Fk}} and subtracting the same term, we obtain
\begin{align*}
    \mE[F_k]&= -\mE\left[\fprod{\tn_xf(x_k,{y}_{k+1})-\nabla_x f(x_k,{y}_{k+1}),~ x_{k}-x_* -\tau \nabla_x f(x_k,{y}_{k+1})}\right]\\
    &\quad\quad+\tau \mE\left[\Vert \tn_xf(x_k,{y}_{k+1})-\nabla_x f(x_k,{y}_{k+1}) \Vert^2\right] +\mE\left[\fprod{\tilde{q}_k-q_k,~y_{k}-y_*}\right] + \sigma \mE\left[\Vert \tilde{q}_k-q_k\Vert^2\right].
\end{align*}
\bc{By Assumption \ref{assump: noise-properties}, the estimates $\tn_x f(x_k, y_{k+1})$ and $\tilde{q}_k$ are unbiased estimates for $\nabla_x f(x_k, y_{k+1})$ and $q_k$ and are independent from the iterates $(x_k,y_k)$; therefore, the expectation of the inner products above are zero. This leads to}, 
\begin{equation}\label{eq:EFk-bound}
    \mE[F_k]=\tau \mE[\Vert \tn_xf(x_k,{y}_{k+1})-\nabla_x f(x_k,{y}_{k+1}) \Vert^2]+\sigma \mE[\Vert \tilde{q}_k - q_k \Vert^2].
\end{equation}
Furthermore, we can bound the last term using the inequality $(a+b)^2\leq 2(a^2+b^2)$ as follows:
{\small
\begin{align*}
\mE[\Vert \tilde{q}_k-q_k\Vert^2]&=\mE[\Vert (1+\theta)(\tn_yf(x_k,y_k)-\nabla_y f(x_k,y_k))-\theta (\tn_y f(x_{k-1},y_{k-1})-\nabla_y f(x_{k-1},y_{k-1}))\Vert^2] \\
&\leq 2\left((1+\theta)^2 \mE[\Vert\tn_yf(x_k,y_k)-\nabla_y f(x_k,y_k)\Vert^2 +\theta^2 \Vert  \tn_y f(x_{k-1},y_{k-1})-\nabla_y f(x_{k-1},y_{k-1})\Vert^2]\right)\\
&=2\big((1+\theta)^2+\theta^2\big)\delta_{(2)}^2.
\end{align*}}%
Thus,
\begin{align}
\label{eq:Fk-bound}
\mE[\bc{F_k}]\leq \Big(\tau+2\sigma\big((1+\theta)^2+\theta^2\big)\Big)\delta_{(2)}^2,\quad \forall~k\geq 0.
\end{align}
Combining this bound with \eqref{ineq: var-bound-sapd} gives the desired result.}
\end{proof}
\sa{The proof of Proposition~\ref{prop: conv-of-SAPD} immediately follows from Lemma~\ref{lem:MI-cond} and Theorem~\ref{thm: theorem-7}.}
\subsection{Proof of Theorem \ref{thm: inv-meas-exists}} 
The proof is based on the Harris' ergodic theorem for Markov chains. Particularly, Hairer and Mattingly have shown in \cite{hairer2011yet} that if the following Condition \ref{cond: drift} and Condition \ref{cond: minorization} hold\footnote{This conditions are also known as the drift and minorization conditions.} for a transition kernel $\bcred{\mathcal{R}^{(\tau,\sigma,\theta)}}$ of a Markov chain \bcred{$\{\xi_k^{(\tau,\sigma,\theta)}\}\in\mathbb{R}^{2d}$}, then the Markov chain admits a unique invariant measure (Theorem 2.1 of \cite{hairer2011yet}). \bcred{Throughout the proof, we drop $\theta$ dependency on the notation and set $\xi_k^{(\tau,\sigma,\theta)}\rightarrow \xi_{k}$ to denote the Markov chain $\{\xi_k^{(\tau,\sigma,\theta)}\}$ generated by SAPD, for simplicity.} 

\begin{condition}\label{cond: drift}
There exists a function $\mathcal{V}:\mathbb{R}^{2d}\rightarrow [0,\infty]$ and constants $K\geq 0$ and $\zeta \in (0,1)$ such that 
$$
(\mcRtdt\mathcal{V})(\xi) \leq \zeta \mathcal{V}(\xi) + K,\quad\forall~\xi \in \mathbb{R}^{2d}.
$$
\end{condition}
\begin{condition}\label{cond: minorization} 
There exists a constant $\varphi \in (0,1)$ and a probability measure $\nu$ \sa{on $\mathcal{B}(\mathbb{R}^{2d})$} such that 
$$ 
\inf_{\xi\in \mathcal{C}_R} \mcRtdt(\xi, A) \geq \varphi \nu(A),\quad\forall~A\in \mathcal{B}(\mathbb{R}^{2d}),
$$
for some $R> 2K/(1-\mg{\zeta})$, where $\mathcal{C}_{R}\sa{\triangleq}\{\xi \in\mathbb{R}^{2d} : \mg{\mathcal{V}}(\xi)\leq R\}$, $K$ and \mg{$\zeta$} are the constants from Condition \ref{cond: drift}. 
\end{condition}
Our proof relies on showing that the Markov kernel associated with the SAPD iterations satisfies Condition \ref{cond: drift} and Condition \ref{cond: minorization} \mg{with appropriate constants $R$, $K$ and $\zeta$}; which will imply that the Markov chain admits a unique invariant measure \mg{based on \cite{hairer2011yet}}. Firstly, we show the Condition \ref{cond: drift} is satisfied for $\mcRtdt$ 
\sa{using 
Theorem~\ref{thm: theorem-7}.} 

We first observe that the inequality \eqref{ineq: lyapunov-decay} is valid for any initialization $(x_0, y_0)$. 
Secondly, this inequality holds if the expectation is conditioned on the previous iterate since the noise on the gradient is independent from the iterates. Consequently, we obtain the following inequalities:
{\small
\begin{subequations}
\label{ineq: drift-ineqs}
\begin{align}
\mathbb{E}\left[\frac{1}{2\tau}\Vert x_2-x_*\Vert^2 + \frac{1-\alpha\sigma}{2\sigma}\Vert y_2-y_*\Vert^2 \mid (x_1,y_1)\right] \leq \frac{\sa{\theta}}{2} \left[ \sa{\frac{1}{\tau}}\Vert x_1-x_*\Vert^2 + \sa{\frac{1}{\sigma}}\Vert y_1-y_*\Vert^2\right]+\sa{\theta~\tilde\Xi_{\tau,\sigma,\theta}\delta_{(2)}^2},\\
\mathbb{E}\left[\frac{1}{2\tau}\Vert x_1-x_*\Vert^2 + \frac{1-\alpha\sigma}{2\sigma}\Vert y_1-y_*\Vert^2\mid (x_0, y_0) \right] \leq \frac{\sa{\theta}}{2} \left[ \sa{\frac{1}{\tau}}\Vert x_0-x_*\Vert^2 + \sa{\frac{1}{\sigma}}\Vert y_0-y_*\Vert^2\right]+\sa{\theta~\tilde\Xi_{\tau,\sigma,\theta}\delta_{(2)}^2}.
\end{align}
\end{subequations}}%
In Lemma \ref{lem: drift-cond}, we provide the conditions for which $\mcRtdt$ satisfies Condition \ref{cond: drift}.
\begin{lemma}\label{lem: drift-cond}
\sa{Suppose Assumptions \ref{assump: obj-fun-prop} and \ref{assump: noise-properties} hold and the SAPD parameters $\tau,\sigma>0$ and $\theta\in(0,1)$ satisfy 
the matrix inequality in~\eqref{cond: MI-2} for some $\alpha \in (0, \frac{1}{\sigma})$ such that $1-\alpha \sigma >\theta$.} 
Let
\begin{equation*}
    \gV(\xi_k)\triangleq\frac{1}{4\tau}(\Vert x_k-x_*\Vert^2+\Vert x_{k-1}-x_*\Vert^2)+ \frac{1-\alpha\sigma}{4\sigma} (\Vert y_k-y_*\Vert^2 + \Vert y_{k-1}-y_*\Vert^2).
\end{equation*}
Then the Markov kernel $\mcRtdt$ of SAPD satisfies the following inequality for any $\xi \in \mathbb{R}^{2d}$,
\begin{equation}\label{ineq: drift}
    (\mcRtdt \gV)(\xi) \leq \mg{\zeta~\gV(\xi) + K},
\end{equation}
\mg{with the constants \sa{$\zeta$ and $K$ chosen as}
\begin{equation} \zeta = \sa{\frac{\theta}{1-\alpha\sigma}}, \quad K = \sa{\theta~\tilde\Xi_{\tau,\sigma,\theta}\delta_{(2)}^2}.
\end{equation}
}
\end{lemma}
\begin{proof}
\sa{The desired result 
immediately follows from the inequalities in~\eqref{ineq: drift-ineqs}. Indeed, we have} 
\begin{align*}
&\mathbb{E}[\gV(\xi_2)| \xi_1]=\mathbb{E}[\frac{1}{4\tau}\left(\Vert x_2-x_*\Vert^2 + \Vert x_1-x_*\Vert^2\right)+\frac{1-\alpha\sigma}{4\sigma}\left(\Vert y_2-y_*\Vert^2+\Vert y_1-y_*\Vert^2\right) \mid \xi_1 ] \\
&\quad\quad\leq \frac{\theta}{4}\left[\sa{\frac{1}{\tau}}(\Vert x_1-x_*\Vert^2 + \Vert x_0-x_*\Vert^2)+ \sa{\frac{1}{\sigma}}(\Vert y_1-y_*\Vert^2+\Vert y_0-y_*\Vert^2)\right] + \sa{\theta~\tilde\Xi_{\tau,\sigma,\theta}\delta_{(2)}^2}. 
\end{align*}
Thus, \sa{since $1-\alpha\sigma\in(0,1)$,} we obtain 
$$
\mathbb{E}[\gV(\xi_2) \mid \xi_1] \leq \sa{\frac{\theta}{1-\alpha\sigma}} \gV(\xi_1) + \sa{\theta~\tilde\Xi_{\tau,\sigma,\theta}\delta_{(2)}^2}.
$$
Because the noise is stationary, the Markov chain is time-homogeneous; hence, the operator $\mcRtdt$ is the same throughout the iterations which gives the desired result. 
\end{proof}

\mg{Based on Lemma \ref{lem: drift-cond}}, \bc{in order to show Condition \ref{cond: minorization}, it is sufficient to show that \mg{given $\varphi \in (0,1)$}, there exists a measure \sa{$\tilde{\nu}$} 
\sa{defined on the Borel $\sigma$-algebra $\mathcal{B}(\mathbb{R}^{d})$} satisfying the inequality
\begin{equation}\label{eq: minorization-nu}
\inf_{\xi\in\mathbb{R}^{2d}} \sa{\{p(\xi, B):\ \gV(\xi)\leq R\}} \geq \varphi~\sa{\tilde{\nu}(B)},\qquad\forall~ B \in \mathcal{B}(\mathbb{R}^{d}),
\end{equation}
for some \sa{$$R>2\theta\delta_{(2)}^2 \tilde\Xi_{\tau,\sigma,\theta}/\left(1-\frac{\theta}{1-\alpha\sigma}\right).$$}%
\mg{This is because if such a measure $\tilde\nu$ 
\sa{defined on $\mathcal{B}(\mathbb{R}^{d})$} exists, then we can always define a measure $\nu$ 
\sa{on $\mathcal{B}(\mathbb{R}^{2d})$} satisfying Condition \ref{cond: minorization} \mg{based on the product measure}.}
More specifically,}
any Borel set in $\mathcal{B}(\mathbb{R}^{2d})$ is of the form $B\times C$ where $B, C\in\mathcal{B}(\mathbb{R}^{d})$ are Borel sets in $\mathbb{R}^d$. Introducing the set 
\sa{$$\cD_R\triangleq\left\{ z\in\mathbb{R}^{d} : \exists z_p  \in\mathbb{R}^{d}\quad \mbox{s.t.}\quad  \gV(\xi) \mg{\leq} R\quad  \mbox{for}\quad \xi=[z^\top z_p^\top]^\top \right\}.$$} 
It can be seen that if \eqref{eq: minorization-nu} holds for \sa{$\tilde\nu$},
then we can define the measure
$$
\mg{\nu(B\times C)\sa{\triangleq}\begin{cases}
\sa{\tilde\nu(B)} & \text{ if } C \mg{\subseteq} \sa{\cD_R}\\
0 & \text{otherwise},
\end{cases}
}
$$ 
which will satisfy the minorization condition, i.e., Condition~\ref{cond: minorization}, on the Markov kernel of $\{\xi_k\}$.

In the next lemma, 
we show that if the distribution function is continuous then for any given constant $\varphi\in(0,1)$, we can find a probability measure $\tilde{\nu}$ and a set of the form $\{\xi \in\mathbb{R}^{2d} | \gV(\xi)\leq R\}$ \sa{for some $R>0$} on which the inequality $p(\xi, A)\geq \varphi \sa{\tilde{\nu}(A)}$ is satisfied for any $A\in\mathcal{B}(\mathbb{R}^d)$.

\begin{lemma}\label{lem: minorization}
Under Assumption \ref{assump-cont-density}, for any $\varphi\in (0,1)$ \mg{fixed}, there exists a positive constant $R>0$ and \mg{a probability measure \sa{$\tilde\nu$} on \sa{$\cB(\mathbb{R}^d)$}} such that for every $z \in \mathbb{R}^d$,
\begin{equation}
\inf_{\xi \in\mathbb{R}^{2d}} \{p(\xi,z):\ \sa{\gV(\xi)\leq R}\} \geq \varphi \tilde\nu(z).
\label{ineq-minorization}
\end{equation} 
\end{lemma}
\begin{proof} \mg{We consider the probability measure}
$$
\tilde\nu(z) \triangleq p(\xi_*,z)\frac{1_{\Vert z-z_*\Vert \leq M}}{\int_{\{\sa{z:}~\Vert z-z_*\Vert\leq M\}}p(\xi_*, z)dz},
$$
where $1_{\Vert z-z_*\Vert\leq M}=1$ if $\Vert z-z_*\Vert \leq M$ and 0 otherwise. 
Suppose  $M>0$ is chosen such that the denominator is not zero. Notice that for sufficiently large $M$, the integral at the denominator can be arbitrarily close to 1. We will show that for any $\varphi \in (0,1)$, there exists positive constants $M$ and $R$ such that \eqref{ineq-minorization} 
holds for every $z$ for this choice of $\tilde\nu$. \mg{In particular, given $\varphi\in(0,1)$, we simply take
$$
M \triangleq  \inf \left\{ m>0: \int_{\sa{z:}~\Vert z-z_*\Vert\leq m}p(\xi_*,z)dz \geq \sqrt{\varphi} \right\},
$$
and 
$$
R \triangleq  \sup \left\{ r>0 : \inf_{\xi\in\mathbb{R}^{2d}:~\gV(\xi)\leq r} p(\xi,z)\geq \sqrt{\varphi}p(\xi_*,z),\quad\sa{\forall~z:}~\Vert z-z_*\Vert \leq M \right\},
$$
then, by definition of $M$ and $R$, it can be seen that 
for every $z$ satisfying $\Vert z-z_*\Vert \leq M$,
\begin{equation}
\label{eq:p-bound}
\inf_{\xi \in \mathbb{R}^{2d}:~\gV(\xi)\leq R} p(\xi, z) \geq \varphi \bc{\tilde\nu}(z); 
\end{equation} 
otherwise, for \sa{any $z$ such that} $\Vert z-z_*\Vert > M$, we have \sa{$\tilde\nu(z) = 0$}, \sa{and \eqref{eq:p-bound}} holds trivially. This implies \eqref{ineq-minorization} and we conclude.
}
\end{proof}
The Proposition \ref{prop: gen-inv-meas-exist} utilizes the results obtained at Lemma \ref{lem: drift-cond} and Lemma \ref{lem: minorization} in order to prove the existence of the invariant measure.

\begin{proposition}\label{prop: gen-inv-meas-exist}
\sa{Under the premise of} Lemma \ref{lem: drift-cond}, \sa{suppose  $\tau,\sigma>0$ and $\theta\in(0,1)$ satisfy 
the matrix inequality in~\eqref{cond: MI-2} for some $\alpha \in (0, \frac{1}{\sigma})$ such that $1-\alpha \sigma >\theta$}. \sa{Let $\varphi \in (0,1)$ and $M>0$ be such that $\int_{z:\Vert z-z_*\Vert <M} p(\xi_*,z)dz\geq\sqrt{\varphi}$} and $R>0$ be such that 
\begin{align}
\inf_{\xi\in\mathbb{R}^{2d}, z\in\mathbb{R}^{d}} \left\{\frac{p(\xi,z)}{p(\xi_*,z)}:\ \sa{\gV(\xi)\bcred{\leq}R,\quad \Vert z-z_*\Vert\leq M}\right\}>\sqrt{\varphi}. 
\end{align}
\sa{If the SAPD parameters $(\tau,\sigma,\theta)$ and the noise variance $\delta_{(2)}^2$ satisfy}
\begin{equation}\label{cond: param-and-variance}
2\sa{\delta_{(2)}^2\theta \tilde\Xi_{\tau,\sigma,\theta}\left(1-\frac{\theta}{1-\alpha\sigma}\right)^{-1}}\bcred{\leq} R,
\end{equation}
where 
$\tilde\Xi_{\tau,\sigma,\theta}$ is defined in Lemma~\ref{lem: drift-cond}, then the SAPD iterates starting from initialization $\xi_0 \in \mathbb{R}^{2d}$ admit a 
unique invariant measure 
$\ptdt_*$, \sa{i.e., 
the distribution of $\xi_k$, $\lambda_{\xi_0}\mcR_k^{(\theta)}$, converges to $\sa{\pi^{(\theta)}_*}$ where $\lambda_{\xi_0}$ denotes the Dirac distribution at $\xi_0$}. Moreover, there exists $C>0$ such that
 \begin{equation}
 \Vert \mcRtdt_k \psi - \ptdt_*(\psi)\Vert \leq C 
 \sa{\left(\frac{\theta}{1-\alpha\sigma}\right)^k} \Vert \psi - \ptdt_*(\psi)\Vert,
 \label{eq-conv-rate-invariant}
 \end{equation}
for every measurable function $\psi$ such that $\Vert \psi \Vert <\infty$, where $\Vert \psi \Vert = \sup_{\xi} \frac{|\psi(\xi)|}{1+\gV(\xi)}$ is the weighted supremum norm. 
\end{proposition}
\begin{proof}
\sa{Since $1-\alpha\sigma>\theta$, from Lemma~\ref{lem: drift-cond}, we get {\bf Condition 1}, i.e.,}
$$
\mcRtdt \gV(\xi) \leq \sa{\frac{\theta}{1-\alpha\sigma}} \gV(\xi) +  \sa{\theta\tilde\Xi_{\tau,\sigma,\theta}\delta_{(2)}^2}.
$$
Therefore, when \eqref{cond: param-and-variance} holds, Lemma~\ref{lem: minorization} implies {\bf Condition 2}, i.e.,
$$
\inf_{\xi \in \mathbb{R}^{2d}} \{p(\xi, z):\ \gV(\xi)\leq R\} \geq \varphi \bc{\tilde{\nu}}(z). 
$$
Therefore, we can conclude from [Theorem 1.2., \cite{hairer2011yet}] that there exists an invariant measure $\ptdt_*$ to which $\mcRtdt_k$ converges \mg{according to \eqref{eq-conv-rate-invariant}}.
\end{proof}
\begin{remark}
\bc{In Proposition \ref{prop: gen-inv-meas-exist}, the constant $C$ is not explicitly available. However, explicit constants can be provided if the convergence is analyzed in a different metric. In particular, if the 
parameters satisfy the inequality 
$1>\frac{\theta}{1-\alpha\sigma}+\frac{2\theta \tilde\Xi_{\tau,\sigma,\theta}\delta_{(2)}^2}{R}\triangleq \gamma_{\sa{R}}^{\tau,\delta,\theta}$, then 
for any $\varphi_0\in (0,\varphi)$,
$$
\bcred{\vertiii{\mcRtdt_k \psi - \ptdt_*(\psi)} \leq \bar{\rho}~\vertiii{\psi - \ptdt_*(\psi)},}
$$
\sa{holds with $\bar{\rho}\sa{\triangleq}1-\min\left\{ \varphi-\varphi_0,\  \frac{R\beta}{2+R\beta}(1-\gamma^{\tau,\delta,\theta}_{R})\right\}$ and} $\beta=\frac{\varphi_0}{\theta \tilde\Xi_{\tau,\delta,\sigma}\delta^{2}_{(2)}}$ for every measurable function $\psi$ 
\sa{such that} \bcred{$||| \psi||| <\infty$}, where \bcred{$\vertiii{\psi}\triangleq\inf_{c\in\mathbb{R}}\Vert \psi+c\Vert_\beta$} and $ \Vert \psi\Vert_\beta \triangleq\sup_{\xi}\frac{|\psi(\xi)|}{1+\beta V(\xi)}$ is the $\beta$-weighted supremum norm.
}
\end{remark}

Lastly, notice the the choice of parameters \sa{given in \eqref{eq:SAPD-parameters} together with $\alpha=\frac{1-\theta}{2\sigma}$} satisfy the condition $1-\alpha \sigma = 1-\frac{1-\theta}{2}\geq \theta$. Therefore, we conclude that the assumption of Lemma \ref{lem: drift-cond} holds with $\Xi_{\theta}\triangleq \Xi_{\tau_\theta,\sigma_\theta,\theta}$; hence, Theorem \ref{thm: expansion-at-inv-meas-mean} directly follows from the Proposition \ref{prop: gen-inv-meas-exist}.
\subsection{Proof of Lemma \ref{lem: finite-moment}} 
Throughout the proof, $\mE[\Vert z_k-z_*\Vert^p]$ and $\mE[\Vert z-z_*\Vert^p]$ denote the 
\sa{$p$-th moments of $z_k-z_*$ when $z_k\sim\lambda\mcR_k$, where $z_0\sim\lambda$ and $z_{-1}\triangleq[x_{-1}^\top~y_{-1}^\top]^\top\sim\lambda$ for $\lambda$ being the initial distribution, and the $p$-th moment of $z-z_*$ when $z\sim \pi_*$ the invariant measure,} respectively. 

We first recall \sa{the inequality in \eqref{ineq: var-bound-sapd} from the proof of Theorem \ref{thm: theorem-7}, i.e.,}
\begin{equation}\label{ineq: bound-on-variance}
\frac{1}{2\theta \tau}\Vert x_N-x_*\Vert^2+ \frac{1-\alpha \sigma}{2\theta \sigma}\Vert y_N-y_*\Vert^2 \leq \theta^{N-1} \left(\frac{1}{2\tau}\Vert x_0-x_*\Vert^2 + \frac{1}{2\sigma}\Vert y_0-y_*\Vert^2 + \sum_{k=0}^{N-1}\theta^{-k} F_k\right),
\end{equation}
where $F_k$ is defined in \eqref{eq:Fk}.

\sa{This inequality holds for all $N\geq 1$ whenever $\tau,\sigma>0$ and \sa{$\theta \in (0,1)$ satisfy \eqref{cond: MI-2}
for some $\alpha \in [0,\frac{1}{\sigma}]$.}
Since Proposition~\ref{prop: conv-of-SAPD} shows that choosing $(\tau,\sigma,\theta)$ as in \eqref{eq:SAPD-parameters} and $\alpha=\frac{1-\theta}{2\sa{\sigma}}<\frac{1}{\sigma}$ satisfies \eqref{cond: MI-2}, we can conclude that for all $N\geq 1$, \eqref{ineq: bound-on-variance} is valid for $(\tau,\sigma,\theta)$ as in \eqref{eq:SAPD-parameters} and $\alpha=\frac{1-\theta}{2\sa{\sigma}}$, which yields the following inequality  for $N=1$,}  
\begin{align}
\label{eq:simple-recursion-1}
\frac{ \mu_x}{2}\Vert x_1-x_*\Vert^2 + \frac{\mu_y(1+\theta)}{4} \Vert y_1-y_*\Vert^2 \leq  \frac{\mu_x\theta}{2} \Vert x_0-x_*\Vert^2 + \frac{\mu_y \theta }{2}\Vert y_0-y_*\Vert^2 +  (1-\theta) F_0.
\end{align}
Notice that the inequality $\theta< \frac{2\theta}{1+\theta}$ holds for all $\theta<1$; \sa{therefore, \eqref{eq:simple-recursion-1} implies that} 
\begin{equation}\label{ineq: bound-on-var-2} 
\frac{\mu_x}{2}\Vert x_1-x_*\Vert^2 + \frac{\mu_y(1+\theta)}{4} \Vert y_1-y_*\Vert^2 \leq  \left( \frac{2\theta}{1+\theta}\right)\left[\frac{\mu_x}{2} \Vert x_0-x_*\Vert^2 + \frac{\mu_y (1+\theta) }{4}\Vert y_0-y_*\Vert^2\right] + (1-\theta) F_0.
\end{equation}
\sa{The selection of the parameters $\theta\in (\hat\theta,1)$ and $\tau,\sigma>0$ in \eqref{eq:SAPD-parameters} imply that $\tau=\mathcal{O}(1-\theta)$ and $\sigma=\mathcal{O}(1-\theta)$.
Furthermore,} Assumption~\ref{assump: noise-properties} \sa{says} that the gradient estimates are \sa{conditionally} independent from the sequence $\{x_k,y_k\}$ generated by SAPD; therefore, if we take the expectation of $F_0$, 
\sa{it follows from \eqref{eq:Fk-bound} that} 
\begin{align}
\label{eq:F0-bound}
    \mE[\mg{F_0}]\leq \Big(\tau+2\sigma\big((1+\theta)^2+\theta^2\big)\Big)\delta_{(2)}^2=\mathcal{O}\Big((1-\theta)\delta_{(2)}^2\Big),
\end{align}
with the convention that \sa{$z_0\sim\pi_*$ and $z_{-1}\sim\pi_*$ are random variables drawn from invariant measure $\pi_*$. 
} 
\sa{Since $\pi_{*}$ is the invariant measure, \mg{we have $z_1 \sim \pi_*$ as well and therefore}  $\mE[\Vert x_1-x_*\Vert^2]= \mE[\Vert x_0-x_*\Vert^2]$ and $\mE[\Vert y_1-y_*\Vert^2]= \mE[\Vert y_0-y_*\Vert^2]$; moreover, Theorem~\ref{thm: theorem-7} implies that these quantities are finite.}
\sa{Let $\bar{\mu}\triangleq\min\{\mu_x,\mu_y\}$}. Hence, taking expectation of both sides of \eqref{ineq: bound-on-var-2} with respect to the invariant measure \sa{$\pi_*$ and using \eqref{eq:F0-bound}}
yields the desired result  
\begin{align}
\label{eq:moment2-bound}
\mE[\Vert z-z_*\Vert ^2 ] \leq \mg{\frac{4}{\bar{\mu}}}
\mE[F_0] = \sa{\mathcal{O}\Big((1-\theta)\delta_{(2)}^2\Big)}.
\end{align}
Next, we will show $\mE[\Vert z-z_*\Vert^4]=\mathcal{O}\left((1-\theta)^2\right)$. \mg{For this purpose}, we are going to bound $|F_0|$. Let $\tilde{y}_k\triangleq y_{k+1}$, \sa{and consider $F_k$ and $q_k$ defined in~\eqref{eq:Fk} and \eqref{def-qk}.} Notice that the Cauchy-Schwarz and triangular inequalities yield for any $k\geq 0$,
\begin{align}\label{ineq: Fk-bounds}
    |F_k| &\leq \Vert \tn_x f(x_k,\tilde{y}_k)-\nabla_x f(x_k,\tilde{y}_k)\Vert \Vert x_k-x_*-\tau \nabla_x f(x_k,\tilde{y}_k)\Vert +\tau \Vert \tn_x f(x_k,\tilde{y}_k)-\nabla_x f(x_k,\tilde{y}_k)\Vert^2\nonumber\\
    &\quad\quad\quad\quad\quad\quad\quad\quad +|(\tilde{q}_k-q_k)^\top (y_k-y_*+\sigma q_k)|+ \sigma\Vert \tilde{q}_k-q_k\Vert^2 \nonumber\\
    &\leq \Vert \tn_x f(x_k,\tilde{y}_k)-\nabla_x f(x_k,\tilde{y}_k)\Vert\Vert x_k-x_*\Vert+ \tau \Vert \tn_xf(x_k,\tilde{y}_k)-\nabla_x f(x_k,\tilde{y}_k)\Vert\Vert \nabla_x f(x_k,\tilde{y}_k)\Vert\nonumber\\
    &\quad\quad  + 
    \sa{\tau}\Vert \tn_x f(x_k,\tilde{y}_k)-\nabla_x f(x_k,\tilde{y}_k)\Vert^2 + \Vert \tilde{q}_k-q_k\Vert \Vert y_k-y_*\Vert +\sigma \Vert \tilde{q}_k-q_k\Vert \Vert q_k\Vert + \sigma \Vert \tilde{q}_k-q_k\Vert^2. 
\end{align}
\sa{Next, we derive some bounds that will be used later to upper bound $\mathbb{E}[|F_0|^2]$. 
First, using $\nabla f(x_*,y_*)=\mathbf{0}$ and \mg{the definition of $q_k$ from \eqref{def-qk}}, we get}
\begin{small}
\begin{align}
\label{eq:qk-bound}
\Vert q_k\Vert &= \norm{(1+\theta)\big(\grad_y f(x_k,y_k)-\sa{\nabla_y f(x_*,y_*)}\big)-\theta \big(\grad_y f(x_{k-1},y_{k-1})-\sa{\nabla_y f(x_*,y_*)}\big)} {\bc{=}}\mathcal{O}\left( \max_{i\in \{k, k-1\}}\Vert z_i-z_*\Vert \right),
\end{align}
\end{small}
\sa{where we used Assumption~\ref{assump: obj-fun-prop} for the \bc{last} \mg{equality}. Furthermore, recalling the definition of $\tilde{q}_k$ from \eqref{alg: sapd}, we also have}
\begin{small}
\begin{align}
\label{eq:tqk_qk-bound}
\Vert \tilde{q}_k-q_k\Vert &=\mathcal{O}\left(\max_{i\in\{k, k-1\}} \{ \Vert \tn_y f(x_i,y_i)- \nabla_y f(x_i,y_i)\Vert\}\right),
\end{align}
\end{small}%
\mg{together with the fact that $y_{k+1}=y_k+\sigma\tilde q_k$} \sa{implies that}
\begin{small}
\begin{align*}
    \Vert \tilde{y}_k - y_*\Vert &\leq \Vert y_k -y_* +\sigma q_k \Vert +\sigma \Vert \tilde{q}_k- q_k\Vert\\ &
    {=} \mathcal{O}\left(\max_{i\in \{k,k-1\}}\{\Vert z_i-z_*\Vert\}\right)+\sa{\sigma}\mathcal{O}\left(\max_{i\in \{k,k-1\}}\{ \Vert \tn_y f(x_i,y_i)-\nabla_y f(x_i,y_i)\Vert\}\right),
\end{align*}
\end{small}%
\mg{where we used \eqref{eq:qk-bound}. Consequently,} \sa{we also get}
\begin{small}
\begin{align}
\label{eq:gradxf_xk_tyk}
    \Vert\nabla_x f(x_k,\tilde{y}_k)\Vert &=\Vert \nabla_x f(x_k,\tilde{y}_k)-\nabla_x f(x_*,y_*)\Vert\nonumber\\
    &= \mathcal{O}\left(\max_{i\in\{k,k-1\}}\{\Vert z_i-z_*\Vert\}\right)+\sa{\sigma}\mathcal{O}\left( \max_{i\in \{k,k-1\}}\{ \Vert \tn_y f(x_i,y_i)-\nabla_y f(x_i,y_i)\Vert\}\right),
\end{align}
\end{small}
\sa{where we again used Assumption~\ref{assump: obj-fun-prop} in the last equality.}

\sa{Suppose the distribution of $\xi_0=(z_0,z_{-1})$ is the stationary distribution of the Markov chain, where $z_{0}=(x_{0},y_{0})$ and $z_{-1}=(x_{-1},y_{-1})$. Setting $k=0$ in \eqref{ineq: Fk-bounds} and using the inequality $(\sum_{i=1}^N\gamma_i)^2\leq N\sum_{i=1}^N\gamma_i^2$ for any $\{\gamma_i\}_{i=1}^N\in\reals^N$, we can bound $\mathbb{E}[|F_0|^2]$ based on the tower law of expectation and the fact that gradient noise is conditionally independent given the gradient arguments as follows:} 
\begin{multline*}
    \mE[|F_0|^2] \leq 6\delta_{(2)}^{2}\mE[\Vert x_0-x_*\Vert^2]+ 6\tau^2 \delta_{(2)}^2\mE[\Vert \nabla_x f(x_0,\tilde{y}_0)\Vert^2]+6 \sa{\tau^2} 
    \delta_{(4)}^4 \\
    + 6 \bc{(1+\theta)^2}\delta_{(2)}^{2} \mE[\Vert y_0-y_*\Vert^2] + 6\sigma^2 \bc{(1+\theta)^2}\delta_{(2)}^{2} \mE[\Vert q_0\Vert^2] + 6\sigma^2 \delta_{(4)}^{4}.
\end{multline*}

\sa{Thus, combining the terms, we get
\begin{align*}
    \mE[|F_0|^2] \leq 6\left( \mE[\Vert z-z_*\Vert^2]\delta_{(2)}^{2}+\tau^2 \mE[\Vert \nabla_x f(x_0,\tilde{y}_0)\Vert^2]\delta_{(2)}^2+\sigma^2 \mE[\Vert q_0\Vert^2]\delta_{(2)}^{2}+(\tau^2+\sigma^2)\delta_{(4)}^{4}\right).
\end{align*}}%
\sa{Using \eqref{eq:F0-bound}, \eqref{eq:qk-bound} and \eqref{eq:gradxf_xk_tyk} together with $\tau=\cO(1-\theta)$ and $\sigma=\cO(1-\theta)$, we get}
\begin{align*}
    \mE[|F_0|^2] = \mathcal{O}\left( \mE[\Vert z-z_*\Vert^2] \right)+\mathcal{O}\left( (1-\theta)^2\right)+\mathcal{O}\left((1-\theta)^2 \mE[\Vert z-z_*\Vert^2]\right)= \mathcal{O}(1-\theta),
\end{align*}
after neglecting the terms with 
\sa{second or} higher orders of $(1-\theta)$ as $\theta \rightarrow 1$.

\mg{On the other hand, }\sa{for $\tau= \frac{1-\theta}{\theta\mu_x}$, $\delta = \frac{1-\theta}{\theta \mu_y}$, $\alpha = \frac{1-\theta}{2\delta}$, and $\bar{\mu}=\min\{\mu_x, \mu_y\}$, \eqref{ineq: bound-on-variance} implies that} 
\begin{align}\label{ineq: bound-on-var-3}
    \frac{\bar{\mu}}{2}\Vert z_N-z_*\Vert^2 \leq \theta^{N} \left( \mu_x \Vert x_0-x_*\Vert^2 + \mu_y\Vert y_0-y_*\Vert^2\right)+ 2\theta^{N-1}(1-\theta)\sum_{k=0}^{N-1}\theta^{-k}F_k,
\end{align} 
\sa{for all $N>0$.} 
Therefore taking the square of both sides of the inequality above yields 
\begin{align*}
\frac{\bar{\mu}^2}{4} \Vert z_N-z_*\Vert^4 &\leq 2\theta^{2N}\left( \mu_x \Vert x_0-x_*\Vert^2 + \mu_y \Vert y_0 - y_*\Vert^2\right)^2 + 4 \theta^{2N-2} (1-\theta)^2\left( \sum_{k=0}^{N-1}\theta^{-k}F_k \right)^2.
\end{align*}
We can use Cauchy-Schwarz inequality \sa{to get} $(\sum_{k=0}^{N-1}\theta^{-k}F_k)^2 \leq \left( \sum_{k=0}^{N-1}\theta^{-k}\right)\left(\sum_{k=0}^{N-1}\theta^{-k}F_k^2\right)$ \sa{implying}
$$
\frac{\bar{\mu}^2}{4} \mE[\Vert z_N-z_*\Vert^4]\leq 2\theta^{2N}\left( \mu_x \Vert x_0-x_*\Vert^2 + \mu_y \Vert y_0 - y_*\Vert^2\right)^2+ 4 
\sa{\theta^N}(1-\theta) \sum_{k=0}^{N-1}\theta^{-k}\mE[F_k^2].
$$
This inequality holds for every initialization $(x_0,y_0)$. Notice that the inequality \eqref{ineq: Fk-bounds} holds 
\sa{in a.s. sense} for any $k\geq0$; hence, \sa{\eqref{eq:gradxf_xk_tyk} implies that} $\mE[|F_k|^2]=\mathcal{O}(\mE[\max_{i\in\{k,k-1\}}\Vert z_i-z_*\Vert^2])$.  Therefore, if we assume the initial distribution of the Markov chain $\{x_k,y_k\}_{k\in\mathbb{N}}$ is Dirac measure at $(x_0,y_0)$ \sa{and $(x_{-1},y_{-1})=(x_0,y_0)$}, then taking the expectation of both sides gives the following inequality for each $N\geq 0$, 
\begin{align*}
\frac{\bar{\mu}^2}{4} \mE[\Vert z_N-z_*\Vert^4]= 2\theta^{2N} \left( \mu_x \Vert x_0-x_*\Vert^2 + \mu_y \Vert y_0 - y_*\Vert^2\right)^2+ 4\sa{\theta}(1-\theta^{N}) \mathcal{O}(\max_{\sa{k\in \{0,\ldots,N-1\}}}\mE[\Vert z_k-z_*\Vert^2]).
\end{align*}
On the other hand, we know that $\{\mE[\Vert z_k-z_*\Vert^2]\}_{k\in \mathbb{N}}$ is a uniformly bounded sequence from Theorem \ref{thm: theorem-7}; therefore, Theorem \ref{thm: inv-meas-exists} suggests $\mE[\Vert z_N-z_*\Vert^4]\rightarrow \mE[\Vert \bcred{z_\infty}-z_*\Vert^4]$ as $N\rightarrow \infty$ \sa{with $\bcred{z_\infty}$ being distributed according to the stationary distribution}. This also implies 
$\mE[\Vert \bcred{z_\infty} -z_* \Vert^4]$ is bounded. In the rest of the proof, \sa{to show $\mE[\Vert \bc{z_\infty} -z_*\Vert^4]=\mathcal{O}((1-\theta)^2)$, we follow a similar approach we used for deriving a bound on the 
2-nd moment of $z-z_*$ when $z\sim\pi_*$ follows 
the stationary distribution, i.e., $\mE[\Vert \bcred{z_\infty} -z_* \Vert^2]$.}

Firstly, we define the function $V(\sa{z})\triangleq\frac{\mu_x}{2}\Vert \sa{x}-x_*\Vert^2 + \frac{\mu_y(1+\theta)}{
\sa{4}}\Vert \sa{y}-y_*\Vert^2$ 
\sa{for $z=(x,y)\in\reals^d$,} which satisfies the following inequality \sa{due to} \eqref{eq:simple-recursion-1}:
$$
V^2(z_1)\leq \left(\frac{2\theta}{1+\theta}\right)^2 V^2(z_0) + \frac{4\theta (1-\theta)}{1+\theta} V(z_0)F_0+ (1-\theta)^2F^2_0.
$$
Let $z_0$ be drawn \sa{from the stationary distribution} $\pi_*$, then we have $\mE[V^2(z_1)]=\mE[V^2(z_0)]$ since $\pi_*$ is the invariant measure for $\{z_\bcred{k}\}_{k\in\mathbb{N}}$. Notice that $\mE[V(z)]=\mathcal{O}(\mE[\Vert z-z_*\Vert^2])$ which is bounded by the arguments above; hence, we obtain,
\begin{align}
\label{eq:V2-bound}
\mE[V^2(\sa{z_0})] \left(\frac{1+3\theta}{1+\theta}\right) \leq  4\theta \mE[V(\sa{z_0})F_0] + (1-\theta^2)\mE[F_0^2].
\end{align}
\sa{As we obtained \eqref{eq:EFk-bound}, we can also get}
\begin{align*}
\mE[F_0 | z_0]&=-\mE[\fprod{\tn_x f(\sa{x_0},\sa{y_{1}})-\nabla_x f(x_0,y_{1}),~x_0-x_*-\tau \tn_x f(x_0,y_{1})} | z_0] + \mE[\fprod{\tilde{q}_0-q_0, y_{1}-y_0}|z_0]\\
&= \tau \mE[\Vert \tn_x f(z_0,y_1)-\nabla_x f(z_0,y_1)\Vert^2|z_0] + \sigma \mE[\Vert \tilde{q}_0-q_0\Vert^2 | z_0] = 
\mathcal{O}\left((1-\theta)\delta_{(2)}^2\right),
\end{align*}
\sa{which follows from Assumption \ref{assump: noise-properties} and \eqref{eq:tqk_qk-bound}.}
Moreover, we \sa{also} have $\mE[V(z_0)]=\mathcal{O}\left(\mE[\Vert z_0-z_*\Vert^2] \right)\sa{=\cO(1-\theta)}$. Thus, 
$$
\mE[V(z_0)F_0] = \mE[V(z_0)\mE[F_0|z_0]] = \mathcal{O}\left((1-\theta)\delta_{(2)}^2\right)\mE[V(z_0)]=\mathcal{O}\left((1-\theta)^2 \delta_{(2)}^{2}\right).
$$
\sa{Since we have already shown $\mE[|F_0|^2]=\mathcal{O}(1-\theta)$; the result directly follows from \eqref{eq:V2-bound}, i.e., when $z\sim\pi_*$, it holds that}
\begin{align}
\label{eq:moment4}
\sa{\mE[\Vert z-z_*\Vert^4]=\cO(\mE[V^2(z)])}=\mathcal{O}\left( (1-\theta)^2\right).
\end{align}

\subsection{Proof of Theorem \ref{thm: expansion-at-inv-meas-mean}}
\mg{By Theorem \ref{thm: inv-meas-exists}, the distribution of the iterates $\xi_k^{(\theta)}$ converges to the invariant measure $\pi_*^{(\theta)}$. In particular, if we take  $\psi(\xi) = \xi$ in Theorem \ref{thm: inv-meas-exists}, then it follows that
 $$ \bar{\xi}^{(\theta)} = \lim_{k}  \mcR_k^{(\theta)}\psi = \lim_{k}\mathbb{E} [\xi_k^{(\theta)}] = \mathbb{E}[\xi_\infty],$$
where $\xi_\infty$ is a random variable distributed according to the invariant measure $\pi_*^{(\theta)}$. Therefore, for computing $\bar{\xi}^{(\theta)}$, without loss of generality, we can assume that the initialization $\xi_1$ is drawn from $\pi_*$; in which case $\xi_k^{(\theta)}$ will also be distributed according to $\pi_*$ for every $k\geq 1$.}
\begin{lemma}\label{lem: char-mean-z}
Consider the setting of Theorem \ref{thm: expansion-at-inv-meas-mean}. Suppose the initialization $\xi_1$ is distributed according to invariant measure $\pi_{*}^{(\theta)}$. Then, we have
\begin{equation}\label{eq: mean-of-z}
    \mE[z_1-z_*]=\bcred{-\frac{1}{2}} [\nabla^{(2)}f_*]^{-1} \nabla^{(3)}f_* \mE[(z_{1}-z_*)^{\otimes 2}]+\sa{\cO\left((1-\theta)^{3/2}\right)}.
\end{equation}
\end{lemma}
\begin{proof}
Let $\nabla f(z)=\nabla f(x,y)$ for simplicity. \sa{Recall that $\grad f(z_*)=\mathbf{0}$ and} consider the 
\sa{second-}order Taylor expansion of $\nabla f$ with integral remainder of $\nabla f$ around $z_*$,
\begin{equation}\label{eq: Taylor-exp-grad}
\nabla f(z)= \sa{\grad f(z_*)+}\nabla^{(2)}f_* (z-z_*)+\frac{1}{2}\nabla^{(3)}f_*(z-z_*)^{\otimes2}+ \bcred{r_1}(z).
\end{equation}
\sa{Since $f$ has uniformly bounded $4$-th order derivative,} $\bcred{r_1}$ satisfies the condition 
\begin{equation}\label{ineq: R1}
\sup_{z\in\mRd}\left\{\frac{\Vert \bcred{r_1}(z)\Vert}{\Vert z-z_* \Vert^3}\right\}<\infty.
\end{equation}
Therefore, we can write the following approximations for each $x_k,y_k$ \sa{and $\tilde{y}_k\triangleq y_{k+1}$} defined \sa{in} \eqref{alg: sapd}:
\begin{subequations}\label{eq: grad-exp-sys}
\begin{align}
\nabla f(x_k,\tilde{y}_k)&=\nabla^{(2)}f_* \begin{bmatrix}x_k-x_*\\\tilde{y}_k-y_*\end{bmatrix}+\frac{1}{2}\nabla^{(3)}f_*\begin{bmatrix}x_k-x_*\\\tilde{y}_k-y_*\end{bmatrix}^{\otimes 2} +\bcred{r_1}(x_k,\tilde{y}_k),\\
\nabla f(z_k)&=\nabla^{(2)}f_*(z_k-z_*)+\frac{1}{2} \nabla^{(3)}f_*(z_k-z_*)^{\otimes 2}+ \bcred{r_1}(z_k),\\
\nabla f(z_{k-1})&=\nabla^{(2)}f_* (z_{k-1}-z_*)+\frac{1}{2}\nabla^{(3)}f_* (z_{k-1}-z_*)^{\otimes 2}+\bcred{\bcred{r_1}}(z_{k-1}).
\end{align}
\end{subequations}
\sa{Taking expectation of the both sides of the equation \eqref{alg: sapd-dyn-sys} yields $\mE[\xi_2]=M\mE[\xi_1]+N\mE[\tPhi_1]$, which implies that
\begin{align}
\label{eq:simple-system}
\mE[z_2]=\mE[z_1]+
\begin{bmatrix}
-\tau\tilde\grad_x f(x_1,\tilde y_1)\\ 
\sigma(1+\theta)\tilde\grad_y f(z_1)-\sigma\theta\tilde\grad_y f(z_0)
\end{bmatrix}.
\end{align}
Since the measure \sa{$\pi_*^{(\theta)}$} is invariant for $\mathcal{R}$, clearly we have $\mE[z_2]=\mE[z_1]$; thus, \eqref{eq:simple-system} yields the following system of equations:}
\begin{subequations}\label{eq: sys-of-eq-at-inv-meas}
\begin{align}
    \mE[z_2-z_*]&=\mE[z_1-z_*], \label{eq: sys-of-eq-at-inv-meas-a}\\
    \mE[\tn_x f(x_1,\tilde{y}_1)]&=0,\\
    (1+\theta)\mE[\tn_yf(z_1)]&=\theta \mE[\tn_yf(z_{0})] \label{eq: sys-of-eq-at-inv-meas-c}.
\end{align}
\end{subequations}
Define the variables, 
\begin{align}\label{def: grad-noises}
w_1^0\sa{\triangleq}\tn_x f(x_1,\tilde{y}_1)-\nabla_x f(x_1,\tilde{y}_1),\;
w_1^1\sa{\triangleq}\tn_yf(z_1)-\nabla_y f(z_1),\;
w_1^2\sa{\triangleq}\tn_yf(z_0)-\nabla_y f(z_0)
\end{align}
and \sa{introduce new} notations, 
$$
H_{1}\bc{\triangleq}\nabla^{(2)}_{xx}\bc{f(z_*)},\; H_2 \bc{\triangleq}\nabla^{(2)}_{xy}\bc{f(z_*)},\; H_3\bc{\triangleq}\nabla^{(2)}_{yy}\bc{f(z_*)}.
$$
Recall that $\tilde{y}_1=y_1+\sigma \tilde{q}_1$, \bc{and $\tilde{q}_1=(1
+\theta)\tn_y f(z_1)-\theta \tn_y f(z_0)$; therefore, \eqref{eq: sys-of-eq-at-inv-meas-c} yields, 
\begin{equation}
\label{eq:y1ty1}
\mE[\tilde{y}_1]=\mE[y_1+\sigma \tilde{q}_1]=\mE[y_1]+\sigma \mE[(1+\theta)\tn_yf(z_1)]-\sigma \theta \mE[\tn_y f(z_0)]=\mE[y_1].
\end{equation}
}
If we set $\bcred{S}\bc{\triangleq}\mE\left[\begin{bmatrix}x_1-x_*\\ \tilde{y}_1-y_*\end{bmatrix}^{\otimes 2}\right]-\mE[(z_1-z_*)^{\otimes 2}]$, \sa{then using
the Taylor expansions in \eqref{eq: grad-exp-sys} and \eqref{eq:y1ty1}, the expectation of \mg{the elements of } $\tPhi_1\bc{=[\tn_x f(x_1, \tilde{y}_1)^\top, \tn_y f(z_1)^\top, \tn_y f(z_0)^{\top} ]^\top}$ can be written as follows:}
\begin{subequations}\label{eq: taylor-exp-tphi}
\begin{align}
\mE[\tn_xf(x_1,\tilde{y}_1)]&=[H_1, H_2]\mE[z_1-z_*]+\frac{1}{2}\Pi_1 \nabla^{(3)}f_*\mE[(z_1-z_*)^{\otimes2}]+ \frac{1}{2}\Pi_1 \nabla^{(3)}f_* \bcred{S} + \Pi_1\mE[\bcred{r_1}(x_1,\tilde{y}_1)]\\
\mE[\tn_y f(z_1)]&=[H_2^\top, H_3]\mE[z_1-z_*]+\frac{1}{2}\Pi_2 \nabla^{(3)}f_* \mE[(z_1-z_*)^{\otimes 2}]+\Pi_2\mE[\bcred{r_1}(z_1)]\\
\mE[\tn_y f(z_0)]&=[H_2^\top, H_3]\mE[z_0-z_*]+\frac{1}{2}\Pi_2 \nabla^{(3)}f_* \mE[(z_0-z_*)^{\otimes 2}]+\Pi_2\mE[\bcred{r_1}(z_0)],
\end{align}
\end{subequations}
where $\Pi_1$ and $\Pi_2$ are projection matrices defined as \bc{$\Pi_1z_k=x_k$ and $\Pi_2z_k=y_{k}$ for any $z_k =\begin{bmatrix}x_k\\y_{k} \end{bmatrix} \in \mathbb{R}^{2d}$} and we used the fact that \sa{$\mE[w_1^0]=\mE[w_1^1]=\mE[w_1^2]=0$.} 

\bc{
Since $\pi_*^{(\theta)}$ is invariant measure, $\mE[(\xi_2-\xi_*)^{\otimes 2}]=\mE[(\xi_1-\xi_*)^{\otimes 2}]$ which also implies  
$\mE[(z_1-z_*)^{\otimes 2}]=\mE[(z_0-z_*)^{\otimes 2}]$, where the expectations are taken with respect to $\pi_*^{(\theta)}$.} \mg{We also recall that $ \mE[z_1 - z_*] = \mE [z_0 - z_*]$. Therefore, the Taylor expansion of the partial derivatives \eqref{eq: taylor-exp-tphi} together with the conditions in~\eqref{eq: sys-of-eq-at-inv-meas} gives the following equalities:}
\begin{align*} 
0&= [H_1, H_2] \mE[z_1-z_*] + \frac{1}{2}\Pi_1 \nabla^{(3)}f_*\mE[(z_1-z_*)^{\otimes 2}]+\frac{1}{2}\Pi_1 \nabla^{(3)}f_*\bcred{S}+ \Pi_1\mE[\bcred{r_1}(x_1,\tilde{y}_1)],\\
0&=(1+\theta)\mE[\tn_yf(z_1)]-\theta\mE[\tn_yf(z_0)]\\
&=[H_2^\top,H_3]\mE[z_1-z_*]+\frac{1}{2}\Pi_2\nabla^{(3)}f_*\mE[(z_1-z_*)^{\otimes 2}]+\Pi_2\left( (1+\theta)\mE[\bcred{r_1}(z_1)]-\theta\mE[\bcred{r_1}(z_0)] \right).
\end{align*} 
Introducing \bc{\begin{equation}\bcred{e_1}\triangleq[\bcred{r_1}(x_1,\tilde{y}_1)^{\top}\Pi_1^{\top}+\frac{1}{2}(\Pi_1\nabla^{(3)}f_*\bcred{S})^{\top},\left( (1+\theta)\bcred{r_1}(z_1)-\theta \bcred{r_1}(z_0) \right)^\top\Pi_2^\top]^\top\sa{\in\reals^d},\label{def-e1}
\end{equation}} these inequalities yield 
\begin{equation}
0=\nabla^{(2)}f_* \mE[z_1-z_*]+\frac{1}{2}\nabla^{(3)}f_*\mE[(z_1-z_*)^{\otimes 2}]+ \mE[\bcred{e_1}].
\label{ineq-intermediate}
\end{equation}
This gives the characterization of the gap $\mE[z_1]-z_*$ with respect to variance $\mE[(z_1-z_*)^{\otimes 2}]$. \mg{For completing the proof, \eqref{ineq-intermediate} shows that it suffices to prove that the error term $\mE[\bcred{e_1}] = \sa{\cO\left((1-\theta)^{3/2}\right)}$. In the rest of the proof, we will study the components of the vector $\mE[\bcred{e_1}]$ to show this.} 
\sa{First, note that H\"older's inequality and \eqref{eq:moment4} imply that 
\begin{align}
    \label{eq:moment3}
    \mE[\Vert z_1-z_*\Vert^3]=\cO\left((1-\theta)^{3/2}\right).
\end{align}}%
The Jensen's inequality, \eqref{ineq: R1} \sa{and \eqref{eq:moment3}} imply \begin{equation}\Vert\mE[\bcred{r_1}(z_i)]\Vert \leq \mE[\Vert \bcred{r_1}(z_i)\Vert]\sa{=\mathcal{O}(\mE[\Vert z_1-z_*\Vert^3])=\cO\left((1-\theta)^{3/2}\right)} \quad \mbox{for each}\quad i\in \{0,1\},
\label{def-r1-zi}
\end{equation} 
Note that the remainder term $\Vert \mE[\bcred{r_1}(x_1,\tilde{y}_1)]\Vert$ can be bounded as 
\begin{align}
    \label{eq:r1-xty}
    \Vert \mE[\bcred{r_1}(x_1,\tilde{y}_1)]\Vert &\leq \mE[\Vert \bcred{r_1}(x_1,\tilde{y}_1)\Vert] \nonumber
    = \mathcal{O}(\mg{\mE[\Vert (x_1,\tilde{y}_1)\Vert^3]})
    \\
    &=\mathcal{O}(\mE[(\Vert x_1-x_*\Vert^2 + \Vert y_1-y_*+\sigma \tilde{q}_1 \Vert^{2})^{3/2}])=\mathcal{O}(\mE[(\Vert z_1-z_*\Vert^2 + \sigma^2 \Vert\tilde{q}_1 \Vert^{2})^{3/2}]),
\end{align} 
\mg{where we used \eqref{ineq: R1}}. In order to bound $\Vert \mE[\bcred{r_1}(x_1,\tilde{y}_1)]\Vert$, we now consider $\Vert\tilde{q}_1 \Vert^{2}$.
{\footnotesize
\begin{align}
\Vert\tilde{q}_1\Vert^{2}=&\Vert(1+\theta)(\tn_y f(z_1)-\nabla_y f(z_1))+(1+\theta)(\nabla_y f(z_1)-\nabla_y f(z_*))-\theta (\tn_y f(z_0)-\nabla_y f(z_0))-\theta(\nabla_y f(z_0)-\nabla f(z_*))\Vert^2\nonumber\\
\leq &4(1+\theta)^2\Vert \tn_y f(z_1)-\nabla_y f(\bc{z_1})\Vert^2 +4 \theta^2 \Vert \tn_y f(z_0)-\nabla_y f(\bc{z_0})\Vert^2\nonumber\\ &+\sa{4(1+\theta)^2} \Vert \nabla_y f(z_1)-\nabla_y f(z_*)\Vert^2 + 4\theta^2 \Vert \nabla_y f(z_0)-\nabla_y f(z_*)\Vert^2.\label{eq:tqk-bound}
\end{align}}%
Recall that the function $t^{3/2}$ is convex and monotonically increasing for $t>0$; hence, it satisfies the property $(\frac{a+b+c+d}{4})^{3/2}\leq \frac{1}{4}(a^{3/2}+b^{3/2}+c^{3/2}+d^{3/2})$ for any $a,b,c,d >0$. \sa{Therefore, if both sides of \eqref{eq:tqk-bound} are raised to the power of $3/2$,} then we obtain
\begin{align}
\label{eq:q3}
\mE[\Vert \tilde{q}_1 \Vert^3] = \bc{(1+\theta)^3\mathcal{O}\left( \delta^{3}_{(3)} +\mE[\Vert z_1-z_*\Vert^3]\right)},
\end{align}
\bc{which} follows from \sa{Assumption \ref{assump: obj-fun-prop} and Assumption \ref{assump: noise-properties} -- note} Assumption~\ref{assump: obj-fun-prop}
implies $\Vert \nabla_y f(z_i)-\nabla_y f(z_*)\Vert^2 \leq 2\max\{L_{yx}^2, L_{yy}^2\}\Vert z_i-z_*\Vert^2$ for each $i\in\{0, 1\}$. \sa{Thus, \eqref{eq:r1-xty} and \eqref{eq:q3} together imply that}
$$
\Vert \mE[\bcred{r_1}(x_1,\tilde{y}_1)]\Vert\bcred{=} 
\mathcal{O}\left(\mE[\Vert z_1-z_*\Vert^3]\bc{+ \sigma^{\bcred{3}}(1+\theta)^{3}\delta_{(3)}^{3}}\right).
$$
Recall that $\sigma=\frac{1-\theta}{\mu_y \theta}$; \sa{thus, using \eqref{eq:moment3},} we get 
\begin{equation}
\Vert \mE[\bcred{r_1}(x_1,\tilde{y}_1)]\Vert=\sa{\mathcal{O}\left( (1-\theta)^{3/2}+(1-\theta)^3\delta_{(3)}^{3}\right)}=\mathcal{O}\left( (1-\theta)^{3/2}\right).
\label{eq-order-r-x1-ty1}
\end{equation}

Lastly, let us consider $\Vert \mE[\Pi_1 \nabla^{(3)}f_* \bcred{S}]\Vert$. Recall that $\Pi_1$ is the projection of the vector $\nabla^{(3)}f_* \bcred{S}$ onto its first $d_x$-coordinates; hence we can write $\Vert \Pi_1 \nabla^{(3)}f_* \bcred{S} \Vert^2 = \sum_{l=1}^{\sa{d_x}} (\sa{\sum_{i=1}^d\sum_{j=1}^d}\frac{\partial^3 f(z_*)}{\partial z_i \partial z_j \partial z_l} [\bcred{S}]_{ij})^2 =\mathcal{O}(\Vert \bcred{S} \Vert_F^2)$ where $\Vert .\Vert_F$ is the Frobenius norm and $[\bcred{S}]_{ij}$ is the $(i,j)$-th element of the matrix $\bcred{S}$. The property $\Vert \bcred{S} \Vert_F^2 \leq \Rank(\bcred{S}) \Vert\bcred{S} \Vert^2_2 $ 
\sa{implies that} $\Vert \mE[\Pi_1 \nabla^{(3)}f_* \bcred{S}] \Vert = \mathcal{O}(\Vert \bcred{S} \Vert_2)$. Therefore the eigenvalues of $\bcred{S}$ are relevant for our analysis. 
\sa{Writing} $\bcred{S}$ explicitly as follows,
\begin{align}
\bcred{S} &= \begin{bmatrix} 0_{d_x} & \mE[(x_1-x_*)(\tilde{y}_1-y_*)^{\top}-(x_1-x_*)(y_1-y_*)^\top]\\
\mE[(\tilde{y}_1-y_*)(x_1-x_*)^{\top}-(y_1-\sa{y_*})(x_1-\sa{x_*})^\top] & 
\mE[(\tilde{y}_1-y_*)(\tilde{y}_1-y_*)^\top-(y_1-\sa{y_*})(y_1-\sa{y_*})^{\top}]
\end{bmatrix}\nonumber\\
&=\sigma\begin{bmatrix}
0_{d_x} & \mE[(x_1-x_*)\tilde{q}_1^{\top}]\\
\mE[\tilde{q}_1(x_1-x_*)^{\top}] & \mE[ \tilde{q}_1(y_1-y_*)^{\top}+(y_1-y_*)\tilde{q}_1^\top]
\end{bmatrix}
+\sigma^2 \begin{bmatrix} 0 & 0 \\ 0 & \mE[\tilde{q}_{1}\tilde{q}_1^\top] \end{bmatrix},\label{eq:S-matrix}
\end{align}
\sa{and using Jensen's inequality,} we note $\Vert\mE[(x_1-x_*)\tilde{q}_1^\top]\Vert_2\leq \mE[\Vert (x_1-x_*)\tilde{q}_1^\top\Vert_2]\sa{\leq \mE[\Vert x_1-x_*\Vert\Vert\tilde{q}_1\Vert]}$ \sa{since the spectral norm $\norm{\cdot}_2$ is sub-multiplicative. Next, we compute this expectation using conditioning on $z_1$ and \eqref{eq:tqk-bound}, i.e.,
{\small
\begin{align}
    \mE\Big[\Vert x_1-x_*\Vert\Vert\tilde{q}_1\Vert~\Big|~ z_1\Big]\leq \Vert x_1-x_*\Vert  \left(\mE\Big[\Vert\tilde{q}_1\Vert^2~\Big|~ z_1\Big]\right)^{\tfrac{1}{2}}=\cO\left( \norm{z_1-z_*}(1+\theta)\left(\norm{z_1-z^*}+\norm{z_0-z^*}+\delta_{(2)}\right)\right).
\end{align}}%
Since $z_0\sim\pi_*^{(\theta)}$ and $\pi_*^{(\theta)}$ is the stationary distribution; thus, we also have $z_1\sim\pi_*^{(\theta)}$, which implies $\mE[\norm{z_0-z_*}^2]=\mE[\norm{\bcred{z_1}-z_*}^2]=\cO(1-\theta)$ based on \eqref{eq:moment2-bound}. Therefore, using H\"older's inequality, we get
\begin{equation}
\label{eq:z0z1}
\mE\Big[\norm{z_1-z_*}\norm{z_0-z_*}\Big]\leq\left(\mE[\norm{z_1-z_*}^2]\right)^{\tfrac{1}{2}}~\left(\mE[\norm{z_0-z_*}^2]\right)^{\tfrac{1}{2}}=\cO(1-\theta).
\end{equation}
Finally, using H\"older's inequality one more time, we obtain  $\mE[\norm{z_1-z_*}]=\cO\big(\sqrt{1-\theta}\big)$. This bound together with \eqref{eq:z0z1} implies that $\mE\Big[\norm{x_1-x_*}~\norm{\tilde q_1}\Big]=\cO(\sqrt{1-\theta})$. \mg{Consequently, we have also}
\begin{equation}
    \label{eq:norm-Exq}
    \left\|\mE\left[(x_1-x_*)\tilde{q}_1^\top\right]\right\|_2=\cO(\sqrt{1-\theta}).
\end{equation}}%
\sa{Let $v_1,v_2\in\mRd$ be some arbitrary vectors such that $\norm{v_1}=\norm{v_2}=1$,}
then \sa{using \eqref{eq:S-matrix},} the quadratic form $v^\top \bcred{S} v$ of the vector $v=[v_1^\top,v_2^\top]^\top$ 
\sa{can be written as follows:}
\bc{
\begin{align}\label{ineq: sing-value-ineq}
    |v^\top \bcred{S} v|=&|2\sigma v_1^\top \mE[(x_1-x_*)\tilde{q}_1^\top]v_2+\sigma v_2^\top \mE[ \tilde{q}_1(y_1-y_*)^{\top}+(y_1-y_*)\tilde{q}_1^\top] v_2 + \sigma^2 v_2^\top \mE[\tilde{q}_1\tilde{q}_1^\top]v_2| \nonumber\\
    \bc{\leq}&2\sigma \sa{\norm{\mE[(x_1-x_*)\tilde{q}_1^\top]}_2} 
    + 2\sigma \sa{\norm{\mE\big[ (y_1-y_*)\tilde{q}_1^{\top}\big]}_2} 
    + \sigma^2 \mE[\Vert\tilde{q}_1\Vert^2],
\end{align}}%
\sa{where we used \emph{(i)} $v_1^\top \mE[(x_1-x_*)\tilde{q}_1^\top]v_2\leq \norm{\mE[(x_1-x_*)\tilde{q}_1^\top]}_2$, which follows from the variational definition of the spectral norm, \emph{(ii)} the definition of the spectral norm, i.e.,
\begin{align*}
    v_2^\top \mE\Big[ \tilde{q}_1(y_1-y_*)^{\top}+(y_1-y_*)\tilde{q}_1^\top\Big] v_2
    &\leq \left\|\mE\Big[ \tilde{q}_1(y_1-y_*)^{\top}+(y_1-y_*)\tilde{q}_1^\top\Big]\right\|_2
    \leq 2\left\|\mE\Big[ (y_1-y_*)\tilde{q}_1^{\top}\Big]\right\|_2,
\end{align*}
\emph{(iii)} similarly, $v_2^\top \mE\Big[\tilde{q}_1\tilde{q}_1^\top\Big]v_2\leq \left\|\mE\Big[\tilde{q}_1\tilde{q}_1^\top\Big]\right\|_2\leq \mE\Big[\norm{\tilde{q}_1\tilde{q}_1^\top}_2\Big]\leq \mE\Big[\norm{\tilde{q}_1}_2^2\Big]$. Using the identical arguments for deriving \eqref{eq:norm-Exq}, we also get $\norm{\mE\big[ (y_1-y_*)\tilde{q}_1^{\top}\big]}_2=\cO(\sqrt{1-\theta})$. Finally, from \eqref{eq:tqk-bound}, we get $\mE[\Vert \tilde{q}_1 \Vert^2] = (1+\theta)^2\mathcal{O}\left( \delta^{2}_{(2)} +\mE[\Vert z_1-z_*\Vert^2]\right)$. Therefore, since $\sigma=\cO(1-\theta)$, all these bounds can be combined to result in
$\Vert \bcred{S} \Vert_2=\cO\Big((1-\theta)^{\tfrac{3}{2}}\Big)$.}
\mg{Combining this with \eqref{eq-order-r-x1-ty1} and \eqref{def-r1-zi}, we conclude from the definition \eqref{def-e1} of $e_1$ that} 
$$
\Vert \mE[\bcred{e_1}] \Vert \leq \mE[\Vert \bcred{e_1}\Vert] =\mathcal{O}\left(
\sa{(1-\theta)^{3/2}} \right).
$$
\mg{Plugging this into \eqref{ineq-intermediate}, we conclude}.
\end{proof}
\mg{Equipped with Lemma \ref{lem: char-mean-z}, now we are ready to complete the proof of Theorem \ref{thm: expansion-at-inv-meas-mean}}. We recall that without loss of generality, we assume $\xi_1$ is distributed according to the invariant measure. In this case, the iterates $\xi_k=[z_k^\top, z_{k-1}^{\top}]^\top$ are also distributed according to the invariant measure and we have $\mE[\xi_k-\xi_*]=\mE[\xi_1-\xi_*]$ for every $k$. From the equations  \eqref{eq: sys-of-eq-at-inv-meas}, we conclude that 
$\mE[z_k-z_*]=\mE[z_{k-1}-z_*]$ for each $k\in\mathbb{N}$. 
\mg{Consequently, we have}
\bc{
\begin{equation}
\mE[\xi_1-\xi_*]= \begin{bmatrix} 
 \mE[z_1-z_*]\\
 \mE[z_0-z_*]
\end{bmatrix} =\begin{bmatrix} 
 \mE[z_1-z_*]\\
 \mE[z_1-z_*]
\end{bmatrix}. 
\label{eq-xi-vs-z}
\end{equation}
}
%
%
\mg{Therefore it suffices to characterize  $\mE[(z_1-z_*)$ which itself depends on $\mE[(z_1-z_*)^{\otimes 2}]$ according to Lemma \ref{lem: char-mean-z}.
For this purpose,} we will next study $\mE[(z_1-z_*)^{\otimes 2}]$. 
We first write the second-order Taylor expansion of $\nabla f$ around $z_*$ with a remainder term,
\begin{equation}\label{eq: taylor-exp-grad}
    \nabla f(x,y)=\sa{\nabla f_*+}\nabla^{(2)}f_* \begin{bmatrix} 
    x-x_*\\
    y-y_*
    \end{bmatrix}+\bcred{\bcred{r_2}}(x,y),\quad\sa{\forall~z=(x,y)\in\reals^d.}
\end{equation}
\sa{Since $f$ has uniformly bounded $3$-rd order partial derivatives,} ${r_2}$ satisfies the condition $\sup_{z\in\mathbb{R}^{d}}\{ \frac{\Vert \bcred{\bcred{r_2}}(z)\Vert}{\Vert z-z_*\Vert^2}\}< \infty$. 
\sa{Since $\grad f_*=\mathbf{0}$, for $x_k, y_k$, and $\tilde{y}_k\triangleq y_{k+1}$ defined by SAPD, we have}
\begin{align*}
    \nabla f(x_k, \tilde{y}_k)&= \nabla^{(2)}f_* \begin{bmatrix} 
    x_k-x_*\\
    \tilde{y}_k-y_*
    \end{bmatrix} + \bcred{\bcred{r_2}}(x_k,\tilde{y}_k)\\ 
    \nabla f(x_k,y_k)&=\nabla^{(2)}f_* (z_k-z_*)+ \bcred{\bcred{r_2}}(z_k)\\
    \nabla f(x_{k-1},y_{k-1})&= \nabla^{(2)}f_* (z_{k-1}-z_*)+\bcred{\bcred{r_2}}(z_{k-1}).
\end{align*}
\sa{Using the fact that $\tilde{y}_k-y_*=y_k-y_*+\sigma \tilde{q}_k$ for any $k\geq 0$, we get}
\begin{align}
    \tn_x f(x_1, \tilde{y}_1)&= [H_1, H_2] (z_1-z_*)+ \sigma H_2 \tilde{q}_1+ \Pi_1 \bcred{\bcred{r_2}}(x_1,\tilde{y}_1)+ w_1^0 \\
    \tn_y f(x_1,y_1)&=[H_2^\top , H_3] (z_1-z_*) + \Pi_2 \bcred{\bcred{r_2}}(z_1) + w_1^1 \\ 
    \tn_y f(x_0,y_0)&= [H_2^{\top}, H_3] (z_0-z_*) + \Pi_2 \bcred{\bcred{r_2}}(z_0)+ w_1^2, 
\end{align}
where $H_1=\nabla_{xx}^{(2)}f_*$, $H_2= \nabla_{xy}^{(2)}f_*$, $H_3= \nabla_{yy}^{(2)} f_*$, the projection matrices $\Pi_i$'s are as defined in the proof of Lemma \ref{lem: char-mean-z}, and the variables $w_1^0, w_1^1$ and $w_1^2$ are as given in \eqref{def: grad-noises}. Let 
\begin{align}\label{def: M0-e2-w1}
M_0\triangleq\begin{bmatrix} 
\nabla^{(2)}f_* & 0_{d\times d_x} & 0_{d \times d_y}\\
0_{d_y\times d} & H_2^\top & H_3
\end{bmatrix},\;
\bcred{e_2}\triangleq\begin{bmatrix}
\sigma H_2\tilde{q}_1+\Pi_1 \bcred{\bcred{r_2}}(x_1,\tilde{y}_1)\\
\Pi_2 \bcred{\bcred{r_2}}(z_1)\\
\Pi_2 \bcred{\bcred{r_2}}(z_0)
\end{bmatrix},\; \sa{w_1}\triangleq \begin{bmatrix} 
\bc{w_1^0}\\ 
w_1^1 \\
w_1^2
\end{bmatrix},\; \xi_*\triangleq \begin{bmatrix} x_*\\
y_*\\
x_*\\
y_*
\end{bmatrix},
\end{align}
then using the dynamical system \sa{notation in~\eqref{alg: sapd-dyn-sys}, we get} 
\begin{equation*}
    \xi_2-\xi_* = \underbrace{(M + NM_0)}_{L} (\xi_1-\xi_*)+ N\bcred{e_2} + \bc{N}\sa{w_1}.\label{eq-L-def}
\end{equation*}
where 
\begin{align}
M=\begin{bmatrix} 
I_d & 0_d \\
I_d & 0_d
\end{bmatrix},\quad  
N= \begin{bmatrix}
-\tau I_{d_x} & 0_{d_x\times d_y} & 0_{d_x\times d_y}\\
0_{d_y\times d_x} &  \sigma(1+\theta)I_{d_y} & -\theta \sigma I_{d_y}\\ 
0_{d_x\times d_x} &  0_{d_x\times d_y} & 0_{d_x\times d_y}\\
0_{d_y} & 0_{d_y} & 0_{d_y} 
\end{bmatrix}.
\label{eq-M-and-N}
\end{align}
and 
\begin{align*}
    L\sa{\triangleq} M + NM_0 = \begin{bmatrix} 
    I_d+ A \nabla^{(2)}f_* & B\nabla^{(2)}f_* \\
    I_d  & 0_d 
    \end{bmatrix},
\end{align*}
with the matrices 
\begin{align*}
    A\sa{\triangleq} \begin{bmatrix} 
    -\tau I_{d_x} & 0_{d_{x}\times d_y}\\
    0_{d_y \times d_x} & \sigma(1+\theta) I_{d_y}
    \end{bmatrix}, \; 
    B\sa{\triangleq} \begin{bmatrix} 
    0_{d_x} & 0_{d_x\times d_y}\\ 
    0_{d_y\times d_x} & -\theta \sigma I_{d_y}
    \end{bmatrix}.
\end{align*}
\mg{Using the fact that the noise $w_1$ satisfies $\mE[w_1~|~\xi_1,\tilde{y}_1]=\mathbf{0}$}, we obtain from \eqref{eq-L-def}.
\begin{align}\label{eq: var-xi-eq}
    \mE[(\xi_2-\xi_*)^{\otimes 2}]&= L\mE[(\xi_1-\xi_*)^{\otimes 2}]L^\top + N\mE[\sa{w_1}^{\otimes 2}]N^\top + \Delta, 
\end{align}
where $\Delta$ has the following form: 
\begin{equation}\label{def: matrix-Delta}
    \Delta\bc{\triangleq}\begin{bmatrix} \Delta_1 & \Delta_2 \\
    \Delta_2^\top & \sa{\Delta}_\bcred{3}
    \end{bmatrix}= L\mE[(\xi_1-\xi_*)\bcred{e_2}^\top]N^\top + N \mE[\bcred{e_2} (\xi_1-\xi_*)^\top]L^\top + N \mE[\bcred{e_2} \bcred{e_2}^{\top}]N^\top,
\end{equation}
\sa{for} $\Delta_i \in \mathbb{R}^{d\times d}$ for each $i\in \{1,2,\bcred{3}\}$. 
\sa{Define $\Sigma\triangleq \mE[(\xi_1-\xi_*)^{\otimes 2}]$ and $W\triangleq\mE[\sa{w_1}^{\otimes 2}]$; furthermore, let}
$W_{ij}\triangleq \mE[(w_1^i)(w_1^j)^\top]$ \sa{for $i,j\in\{0,1,2\}$,}
$\Sigma_1\triangleq\mE[(z_1-z_*)^{\otimes 2}]$, $\Sigma_2\triangleq\mE[(z_1-z_*)(z_0-z_*)^\top]$, and $ \;\Sigma_3\triangleq\mE[(z_0-z_*)^{\otimes 2}]=\Sigma_1$.
Firstly notice that 
\begin{align*}
    NWN^\top = \begin{bmatrix}
    \bar{W} & 0_{d}\\
    0_d & 0_d\end{bmatrix} 
\end{align*}
\sa{for $\bar{W}$ such that}
\begin{equation}\label{def: barW}
\bar{W}= T W T^\top, \quad T\triangleq \begin{bmatrix}\sa{-}\tau I_{d_x} & 0_{d_x\times d_y} & 0_{d_x\times d_y}\\
0_{d_y\times d_x} & \sigma (1+\theta)I_{d_y} & -\theta \sigma I_{d_y}
\end{bmatrix}.  
\end{equation}
Therefore, the equation \eqref{eq: var-xi-eq} yields the following system of equations, 
\begin{subequations}\label{eq: sys-cov-matrix}
\begin{align}
    \Sigma_1 & = (I+A\hes)\Sigma_1 \sa{(I+A \hes)^{\top}}+ (I+A\hes)\Sigma_2 \hes B^\top + B \sa{\hes} \Sigma_2^\top (I+A \hes)^{\top}, \nonumber\\
    &\quad\quad\quad\quad\quad + B \hes \Sigma_1 \hes B^\top +\Delta_1 + \bar{W}, \label{eq: sys-cov-matrix-a} \\
    \Sigma_2^{\top}&= \Sigma_1(I+A\hes)^\top + \Sigma_2 \hes B^\top + \Delta_2^\top, \label{eq: sys-cov-matrix-b}\\
    \Sigma_2&= (I+A\hes)\Sigma_1 + B\hes \Sigma_2^\top + \Delta_2,\label{eq: sys-cov-matrix-c}\\
    \sa{\Sigma_1} &= \sa{\Sigma_1+ 0_d+ \Delta_\bcred{3};} \label{eq: sys-cov-matrix-d}
\end{align}
\end{subequations}
\sa{thus, $\Delta_\bcred{3}=0_d$}. If we introduce the matrix $C\sa{\triangleq}A+B = \begin{bmatrix}
-\tau I_{d_x} & 0_{d_x \times d_y}\\
0_{d_y \times d_x} & \sigma I_{d_y}
\end{bmatrix}$, and the notation $H\triangleq\hes $ for simplicity of the presentation, \sa{then using \eqref{eq: sys-cov-matrix-c} within \eqref{eq: sys-cov-matrix-a} we get the first equation below and substituting \eqref{eq: sys-cov-matrix-b} into \eqref{eq: sys-cov-matrix-c} we get the second equation below:}
\begin{subequations}\label{eq: sys-cov-matrix-2}
\begin{align}
    \Sigma_1- BH \Sigma_1 H B^\top &= \Sigma_2 - B H \Sigma_2 H B^\top +\Sigma_2H C^\top + CH \Sigma_2 HB^\top + \Delta_1- \Delta_2(I+ AH)^\top+\bar{W}\\ 
    \Sigma_2 - B H\Sigma_2 H B^\top &= \Sigma_1 - B H \Sigma_1 HB^\top  + CH \Sigma_1 + BH \Sigma_1 HC^\top + BH \Delta_2^\top + \Delta_2.
\end{align}
\end{subequations}
\bcred{Next, we \mg{estimate} $\Sigma_1$ by approximately solving the coupled matrix equations \eqref{eq: sys-cov-matrix-2} known \mg{as coupled Lyapunov equations} (see e.g. \cite{borno1995parallel}). For this purpose, first we are going to study the terms $BH\Sigma_iHB^\top$ and $BH \Sigma_i HC^\top$ and show that their spectral norms are $\mathcal{O}((1-\theta)^3)$ for $i=1,2$. 
\sa{From the definitions of $B$ and $C$, the following} equalities hold for any matrix $X\in \mathbb{R}^{d\times d}$:
\begin{equation}\label{eq: operator-prop-1}
    BHXHB^\top = \sa{(\theta \sigma)^2}\begin{bmatrix}
    0_{d_x} & 0_{d_x\times d_y}\\
    0_{d_y\times d_x} & \sa{\tilde{X}_4}
    \end{bmatrix},
    \qquad  
    BHXH C^\top= \theta \sigma \begin{bmatrix} 0_{d_x\times d_y} & 0_{d_x\times d_y}\\ 
\tau \sa{\tilde{X}_3} & -\sigma \sa{\tilde{X}_4}\end{bmatrix},
\end{equation}
where \sa{$\tilde{X}_3 \in \mathbb{R}^{d_y\times d_x}$} and $\sa{\tilde{X}_4} \in \mathbb{R}^{d_y\times d_y}$ \sa{are defined as the blocks of the matrix $HXH$ satisfying
$
\begin{bmatrix} 
\tilde{X}_1 & \tilde{X}_2 \\ 
\tilde{X}_3 & \tilde{X}_4
\end{bmatrix} = HXH$}. 
\sa{Since in Theorem \ref{thm: expansion-at-inv-meas-mean}, we set $\tau=\frac{1-\theta}{\mu_x\theta}$ and $\sigma = \frac{1-\theta }{\mu_y \theta}$, 
\eqref{eq: operator-prop-1} 
implies that} $\Vert BHXHB^\top\Vert_2=\mathcal{O}((1-\theta)^2 \Vert \sa{\tilde{X}_4}\Vert_2)=\mathcal{O}((1-\theta)^2\Vert X\Vert_2)$ and $\Vert BHXHC^\top\Vert_2= \mathcal{O}((1-\theta)^2\max_{i=\{\sa{3,4}\}}\Vert \tilde{X}_i\Vert_2)=\mathcal{O}((1-\theta)^2\Vert X\Vert_2)$ for each $i\in \{1,2\}$, where we bound the spectral norm of a matrix from the spectral norm of a submatrix (see \cite{sing-val-submatrix} for similiar inequalities). Hence, we get $\Vert BH\Sigma_iHB^\top\Vert_2=\mathcal{O}((1-\theta)^2 \Vert\Sigma_i\Vert_2)$ and $\Vert BH\Sigma_iHC^\top\Vert_2=\mathcal{O}((1-\theta)^2 \Vert \Sigma_i\Vert_2)$ for each $i\in\{1,2\}$. 
\sa{Note Jensen's inequality} implies 
\begin{equation}
    \label{eq:Sigma1-bound}
    \Vert\Sigma_1\Vert_2\leq \mE[\Vert (z_1-z_*)(z_1-z_*)^\top\Vert_2]=
\sa{\mE[\Vert z_1-z_*\Vert^2];}
\end{equation}
similarly, we also have 
\begin{equation}
    \label{eq:Sigma2-bound}
    \Vert \Sigma_2\Vert_2 \leq \mE[\Vert (z_1-z_*)(z_0-z_*)^\top\Vert_2] \sa{=\mE[\norm{z_1-z_*}\norm{z_0-z_*}]} \leq \mE[\Vert z_1-z_*\Vert^2]^{1/2}\mE[\Vert z_0-z_*\Vert^2]^{1/2}
\end{equation} 
by H\"older and Cauchy-Schwarz inequalities. Finally, \sa{since $z_1$ and $z_0$ are distributed according to the stationary distribution, Lemma \ref{lem: finite-moment} implies that $\mE[\norm{z_i-z_*}^2]=\cO(1-\theta)$; therefore, \eqref{eq:Sigma1-bound} and \eqref{eq:Sigma2-bound} imply the} desired results for each $i=1,2$,
\begin{align}\label{eq: prod-apprx}
\Vert BH\Sigma_i HB^\top \Vert_2 = \mathcal{O}\left( (1-\theta)^{3}\right),\qquad  \Vert BH\Sigma_iHC^\top\Vert_2 =\mathcal{O}\left( (1-\theta)^3 \right).
\end{align}

Consequently, we can write the system in \eqref{eq: sys-cov-matrix-2} as follows:
\begin{subequations}
\label{eq:S1S2}
\begin{align}
\Sigma_1&= \Sigma_2 + \Sigma_2 HC^\top+ \Delta_1 - \Delta_2(I+AH)^\top + \bar{W} + \mathcal{O}\left((1-\theta)^3\right)\\
\Sigma_2 &= \Sigma_1+ CH\Sigma_1 +BH\Delta_2^\top + \Delta_2+\mathcal{O}\left((1-\theta)^3\right). 
\end{align}
\end{subequations}
\sa{If we eliminate $\Sigma_2$ within the system in~\eqref{eq:S1S2}, we get the following equation in terms of $\Sigma_1$,}}
\bcred{
\begin{equation*}
    \Sigma_1= \Sigma_1 +CH\Sigma_1+ \Sigma_1 HC^\top +CH\Sigma_1HC^\top +BH\Delta_2^\top+ \Delta_2HB^\top +BH\Delta_2HC^\top+ \Delta_1 + \bar{W} + \mathcal{O}\left( 1-\theta)^3\right)
\end{equation*}
\sa{The arguments we used for deriving \eqref{eq: prod-apprx} also imply that $\Vert CH\Sigma_1HC^\top\Vert_2=\cO((1-\theta)^2\norm{\Sigma_1}_2)$; thus, we get $\Vert CH\Sigma_1HC^\top\Vert_2=\mathcal{O}\left((1-\theta)^3\right)$. Furthermore, \eqref{eq: sys-cov-matrix-c} implies that $\norm{\Delta_2}_2=\cO(\norm{\Sigma_1}_2+\norm{\Sigma_2}_2)$; hence, we also get $\Vert BH \Delta_2 HC^\top \Vert_2=\mathcal{O}\left( (1-\theta)^3 \right)$. Therefore,}}%
\begin{align} \label{eq: Lyapunov-eq}
\mg{ CH\Sigma_1 +\Sigma_1 HC^\top} 
&= \bcred{-}\bar{W}\bcred{-} (\underbrace{\Delta_2 HB^\top + BH\Delta_2^\top +\Delta_1)}_{E}+
 \bcred{\mathcal{O}\left( (1-\theta)^3 \right).}
\end{align}
\bcred{The matrix on the left-hand-side of the equation depends on \mg{$\theta$}. Particularly, if we define the matrix 
\begin{equation}\label{def: C0}
C_0\triangleq \begin{bmatrix} -\frac{1}{\mu_x}I_{d_x} & 0_{d_x \times d_y} \\ 
0_{d_y\times d_x} & \frac{1}{\mu_y} I_{d_y}
\end{bmatrix},
\end{equation}
\sa{then we have} $C=\frac{1-\theta}{\theta}C_0$ and \eqref{eq: Lyapunov-eq} \mg{can be written as} 
\mg{
\begin{equation}
\bcred{\big((C_0 H)\otimes I_d + I_d\otimes(C_0 H) \big) }\Sigma_1 =Q,
 \quad
Q  \triangleq -\left(\frac{\theta}{1-\theta} \bar{W}+ \frac{\theta}{1-\theta}E\right)+\mathcal{O}\left( (1-\theta)^2 \right),
 \label{eq-Lyapunov}
\end{equation}}%
\mg{where $\otimes$ denotes the tensor product \sa{for matrices, i.e., for any matrices $M, N\in \mathbb{R}^{d\times d}$ $M\otimes N:\reals^{d\times d}\to\reals^{d\times d}$ such that $M\otimes N: P \rightarrow MPN^\top$}}.
\sa{The equation in \eqref{eq-Lyapunov}} is known as the Lyapunov equation, well-studied in the control literature. \sa{$\Sigma_1$ satisfying \eqref{eq-Lyapunov} is unique,} 
if the real part of the eigenvalues of the matrix $C_0 H$ are all strictly less than 0, \mg{in which case the solution can be expressed as 
\begin{equation} {\displaystyle \Sigma_1 \triangleq - \int _{0}^{\infty }{e}^{\tau(C_0 H)}~Q~\mathrm {e} ^{\tau(C_0 H)^\top}\;d\tau },
\label{eq-integral}
\end{equation}}%
\mg{(see e.g \cite{zhou1996robust})}. \mg{In this case, the operator $\big(\bcred{(C_0 H)\otimes I_d + I_d\otimes(C_0 H)} \big)$ is invertible with an inverse $\Gamma \triangleq \big(\bcred{(C_0 H)\otimes I_d + I_d\otimes(C_0 H)}\big)^{-1}$, i.e., we have $\Sigma_1 = \Gamma Q$ where $\Gamma$ is the linear (integral) operator defined by \eqref{eq-integral}}. 
In the following lemma, we show that the real part of the eigenvalues of the \mg{operator} $\bcred{(C_0 H)\otimes I_d + I_d\otimes(C_0 H)}$ are indeed strictly less than zero. This result implies that the inverse $\Gamma$ exists.
}


\begin{lemma}\label{lem: lyapunov-op-invertible}
\mg{The matrix $C_0 H$ is stable, i.e., the real parts of the eigenvalues of the matrix $C_0H$ are all negative, where $H= \hes$ is the Hessian of $f$ at the saddle point $z_*$ and $C_0$ is the matrix defined in \eqref{def: C0}.} 
\end{lemma}
\begin{proof}
We introduce the matrix $P=\begin{bmatrix} 
\frac{i}{\sqrt{\mu_x}}I_{d_x} & 0_{d_x\times d_y} \\
0_{d_y\times d_x} & \frac{1}{\sqrt{\mu_y}}I_{d_y}
\end{bmatrix}$ where $i$ is the imaginary unit satisfying $i^2=-1$. Then, $C_0H$ is similar to the matrix $\tilde{H}$ defined as
$$\tilde{H}\bc{\triangleq}
P^{-1}(C_0H) P= 
\bcred{
\begin{bmatrix} 
\frac{-1}{\mu_x} H_1 & \frac{i}{\sqrt{\mu_x\mu_y}}H_2
\\
\frac{\sa{i}}{\sqrt{\mu_x\mu_y}}H_2^\top & \frac{1}{\mu_y}H_3
\end{bmatrix},\quad\text{where}\quad 
H=\begin{bmatrix} H_1 & H_2 \\
H_2^\top & H_3\end{bmatrix}
}
$$
\sa{such that $H_1=\nabla_{xx}^{(2)}f_*$, $H_2= \nabla_{xy}^{(2)}f_*$, $H_3= \nabla_{yy}^{(2)} f_*$.} Hence, $C_0H$ is diagonalizable if and only if $\tilde{H}$ is diagonalizable. \mg{Moreover, $C_0H$ and $\tilde{H}$ share the same eigenvalues as they are similar}. The matrix \sa{$\tilde{H}\in\mathbb{C}^{d\times d}$ can be decomposed into $\tilde H_r\in\reals^{d\times d}$ and $\tilde H_c\in\reals^{d\times d}$, i.e.,} real and imaginary parts, as follows: 
$$
\tilde{H}= \underbrace{\begin{bmatrix} 
\frac{-1}{\mu_x} H_1 & 0_{d_x\times d_y} \\ 
0_{d_y \times d_x} & \frac{1}{\mu_y}H_3
\end{bmatrix}}_{\tilde{H}_r}+ i\underbrace{ \begin{bmatrix} 
0_{d_x} & \frac{1}{\sqrt{\mu_x\mu_y}}H_2\\ 
\frac{1}{\sqrt{\mu_x\mu_y}}H_2^\top & 0_{d_y}
\end{bmatrix}}_{\tilde{H}_c}.
$$
Let $\lambda$ be the eigenvalue of $\tilde{H}$ with the corresponding right eigenvector $v$, i.e., $\tilde{H}v=\lambda v$; \sa{then its complex conjugate, $\bar{\lambda}$, i.e., $\lambda \bar{\lambda}=|\lambda|^2$,} is also an eigenvalue of the \mg{Hermitian transpose (conjugate transpose)} of $\tilde{H}$ with the \mg{left} eigenvector \sa{$v^*$}, i.e., \sa{$v^* \tilde{H}^{*}=\bar{\lambda} v^*$.} 
Therefore, we can write the following for any eigenvalue $\lambda$ of $\tilde{H}$ with the eigenvector $v$ such that \sa{$v^* v=1$}:
\begin{equation}
Re(\lambda)= \sa{\frac{1}{2}\left(\lambda v^* v + \bar{\lambda} v^* v\right)= \frac{1}{2}\left( v^* \tilde{H} v + v^* \tilde{H}^{*} v\right) = v^*\left(\frac{\tilde{H} + \tilde{H}^{*}}{2} \right)  v.} \label{eq-relambda}
\end{equation}
\sa{Since $\tilde{H}$ and $\tilde{H}^{*}$ are symmetric matrices, i.e., $\tilde{H}_{ij} = \tilde{H}_{ji}$ and $\tilde{H}^*_{ij} = \tilde{H}^*_{ji}$, we have $\tilde{H}^{*}=\tilde{H}_r-i\tilde{H}_c$. Therefore, \eqref{eq-relambda} implies that $Re(\lambda)=v^* H_r v$. Now, writing $v=v_r+iv_c$, we obtain
\begin{align}
v^* \tilde{H}_r v&= (v_r-iv_c)^\top \tilde{H}_r(v_r+iv_c)=v_r^\top \tilde{H}_r v_r+ v_c\tilde{H}_r v_c +i (v_r^\top \tilde{H}_r v_c - v_c^\top \tilde{H}_r v_r)= v_r^\top \tilde{H}_r v_r+ v_c\tilde{H}_r v_c,\label{eq-helper}
\end{align}
where we used the fact that $\tilde{H}_r\in\reals^{d\times d}$ is a symmetric matrix. Thus, we can conclude that} 
\begin{equation}
Re(\lambda)= v_r^\top\tilde{H}_rv_r \sa{+} v_c^\top \tilde{H}_r v_c.
\label{eq-re-lambda}
\end{equation}
By Assumption \ref{assump: obj-fun-prop}, \sa{$f(x,y)$} is strongly convex 
\sa{in $x$ for any fixed $y$, and strongly concave 
in $y$ for any fixed $x$; therefore,} we have $-H_1 \prec 0$ and $H_3 \prec 0$, \sa{which implies that $\bcred{\tilde{H}_r}\prec 0$}. We conclude from \eqref{eq-re-lambda} that $Re(\lambda)<0$ for any eigenvalue $\lambda$. 
\end{proof} 
Based on Lemma \ref{lem: lyapunov-op-invertible}, the integral in the representation \eqref{eq-integral} is finite, and the inverse $\Gamma=\bcred{\big((C_0 H)\otimes I_d + I_d\otimes(C_0 H)\big)}^{-1}$ exists where we can write
\begin{equation}\label{eq: Sigma1}
\bcred{\Sigma_1 = \sa{-}\left(\frac{\theta}{1-\theta}\Gamma \bar{W}+ \frac{\theta}{1-\theta} \Gamma E \right)+ \mathcal{O}\left((1-\theta)^2\right)}
\end{equation}
\bcred{\sa{Recall that $\bar{W}= T W T^\top$, where $T$ is defined in \eqref{def: barW} and $W=\mE[w_1^{\otimes 2}]$.} Clearly, $\bar{W}$ depends on $(1-\theta)$. \sa{In order to see this, 
first we decompose $T$ using}
\begin{align*}
T_1\triangleq\begin{bmatrix} 
-\frac{1}{\mu_x} I_{d_x} & 0_{d_x\times d_y} & 0_{d_x\times d_y}\\
0_{d_y\times d_x} & \frac{1}{\mu_y} I_{d_y} & 0_{d_y}
\end{bmatrix},\;\;\quad \;\;
T_2\triangleq
\begin{bmatrix} 
0_{d_x} & 0_{d_x\times d_y} & 0_{d_x\times d_y}\\
0_{d_y\times d_x} & \frac{1}{\mu_y} I_{d_y} & -\frac{1}{\mu_y}I_{d_y}
\end{bmatrix},
\end{align*}
so that
$$
\mg{T =} \begin{bmatrix}\sa{-}\tau I_{d_x} & 0_{d_x\times d_y} & 0_{d_x\times d_y}\\
0_{d_y\times d_x} & \sigma (1+\theta)I_{d_y} & -\theta \sigma I_{d_y}
\end{bmatrix}= \frac{1-\theta}{\theta }T_1 + (1-\theta) T_2.
$$
Then, we can decompose $\bar{W}$ using 
\sa{$\bar{W}= T W T^\top$ as follows:}
\begin{align*}
\bar{W}=\frac{(1-\theta)^2}{\theta^2}W_1 + \frac{(1-\theta)^2}{\theta}W_2 + (1-\theta)^2W_3,
\end{align*}
where $W_1 \triangleq T_1 W T_1^\top$, $W_2 \triangleq T_1 W T_2^\top + T_2 W T_1^\top$ and $W_3 \triangleq T_2 W T_2^\top$, \mg{and the matrices $W_1, W_2$ and $W_3$, \sa{by definition,} do not depend on $\theta$.}
}

Hence, the equation \eqref{eq: Sigma1} becomes:
\begin{align}
\Sigma_1= \sa{-\left(\frac{(1-\theta)}{\theta} \Gamma W_1 + (1-\theta) \Gamma W_2 + \theta(1-\theta)\Gamma W_3 + \frac{\theta}{1-\theta}\Gamma E\right)} + \mathcal{O}\left( (1-\theta)^2 \right).
\label{eq-Sigma1-expansion}
\end{align}
\bcred{Since $\frac{1-\theta}{\theta}=1-\theta+ \mathcal{O}((1-\theta)^2)$ as $\theta\to 1$, \sa{we can write} 
\begin{equation}\label{eq: Sigma_1_ext}
    \Sigma_1= -(1-\theta) \Gamma (W_1+W_2+W_3)  \mg{-}\frac{\theta}{1-\theta}\Gamma E +\mathcal{O}((1-\theta)^2).
\end{equation}
\mg{In the rest of the proof, we analyze the} 
order of $\Vert E\Vert_2$ with respect to $(1-\theta)$. Recall the definition of $E$ from \eqref{eq: Lyapunov-eq}, \sa{i.e.,}
\begin{align}
    E= BH\Delta_2^\top + \Delta_2HB^\top+\Delta_1,
    \label{def-E-recall}
\end{align}
where $\Delta_1$ and $\Delta_2$ are the submatrices of the matrix $\Delta$ defined in \eqref{def: matrix-Delta}, \sa{i.e.,} 
\begin{equation}
\Delta = L \mE[(\xi_1-\xi_*)e_2^\top]N^\top + N\mE[e_2(\xi_1-\xi_*)^\top]L^\top + N \mE[e_2e_2^\top]N^\top.
\label{def-Delta-matrix-2}
\end{equation}
Each matrix $\Delta_i$ is a submatrix of $\Delta$; therefore, their spectral norm is bounded by $\Vert \Delta\Vert_2$, i.e. $\Vert \Delta_i\Vert_2\leq \Vert \Delta\Vert_2$ for each $i\in\{1,2\}$ --see for example \cite{sing-val-submatrix}. Therefore, it is sufficient to bound $\Vert \Delta\Vert_2$ rather than each $\Delta_i$. \mg{We observe from \eqref{def-Delta-matrix-2} that this requires obtaining a bound on the error term $e_2$,} \sa{defined in \eqref{def: M0-e2-w1},}
which \mg{depends on} $\tilde{q}_1$, $r_2(z_0)$, $r_2(z_1)$, and $r_2(x_1,\tilde{y}_1)$. \mg{In particular}, since $\bcred{r_2(.)}$ is 
\sa{the remainder of the first-order Taylor expansion 
of $\grad f(z_i)$ around $z_*$,} we have $$\Vert r_2(z_i)\Vert = \mathcal{O}\left( \Vert z_i-z_*\Vert^2 \right),$$ for each $i\in\{0,1\}$. \sa{Similarly, we have $\Vert r_2(x_1,\tilde{y}_1)\Vert=\mathcal{O}\left(\Vert x_1-x_*\Vert^2 + \Vert \tilde{y}_1-y_*\Vert^2\right)$. Since $\tilde{y}_1=y_1+\sigma\tilde q_1$, we also get} 
\begin{align}\label{ineq: r_2-1}
\Vert r_2(x_1,\tilde{y}_1)\Vert^2
= \mathcal{O}\left((\Vert z_1-z_*\Vert^2 + \sigma^2 \Vert \tilde{q}_1\Vert^2)^2\right)= \mathcal{O}\left(\Vert z_1-z_*\Vert^4 + \sigma^4 \Vert\tilde{q}_1\Vert^4\right),
\end{align}
\mg{where we used the Cauchy-Schwarz inequality and the definition of $\tilde{y}_1$.} 
Consequently, we can bound $\Vert e_2\Vert^2$ using \eqref{ineq: r_2-1}:
\begin{align}
    \Vert \bcred{e_2} \Vert^2 &= \Vert \sigma H_2 \tilde{q}_1+ \Pi_1 \bcred{\bcred{r_2}}(x_1,\tilde{y}_1)\Vert^2+ \Vert \Pi_2 \bcred{\bcred{r_2}}(z_0)\Vert^2+ \Vert \Pi_2 \bcred{\bcred{r_2}}(z_1)\Vert^2 \nonumber\\
    &=\mathcal{O}\left(\sigma^2 \Vert \tilde{q}_1\Vert^2 + \Vert r_2(x_1,\tilde{y}_1)\Vert^2+\Vert z_0-z_*\Vert^4+\Vert z_1-z_*\Vert^4\right)\nonumber\\
    &=\mathcal{O}(\sigma^2 \Vert \tilde{q}_1\Vert^2+\sigma^4\Vert\tilde{q}_1\Vert^4+\Vert z_0-z_*\Vert^4+\Vert z_1-z_*\Vert^4). \label{ineq: e_2}
\end{align}
Notice that the inequality \eqref{eq:tqk-bound} \sa{directly implies bounds on $\mE[\Vert \tilde{q}_1\Vert^2]$ and $\mE[\Vert \tilde{q}_1\Vert^2]$ as given below:}
\begin{align*}
\mE[\Vert \tilde{q}_1\Vert^2]=\mathcal{O}\left((1+\theta)^2(\delta_{(2)}^2+\mE[\Vert z_1-z_*\Vert^2])\right), \; \quad \;\mE[\Vert \tilde{q}_1\Vert^4]=\mathcal{O}\left((1+\theta)^4(\delta_{(4)}^4+\mE[\Vert z_1-z_*\Vert^4])\right),
\end{align*}
\mg{(see the derivation of \eqref{eq:q3} for further details)}.
Therefore, by taking the expectation of both sides of equation \eqref{ineq: e_2}, \mg{and using the fact that $\sigma = \mathcal{O}(1-\theta)$}, we obtain the following relation 
\begin{eqnarray}\label{eq: bound-e2}
\mE[\Vert e_2 \Vert^2] 
&=\mathcal{O}\left( (1-\theta)^2 (\delta_{(2)}^2+\mE[\Vert z_1-z_*\Vert^2)+(1-\theta)^4(\delta_{(4)}^4+\mE[\Vert z_1-z_*\Vert^4])  +
\mg{\mE\Vert z_1-z_*\Vert^4} 
\right), \\ 
&=\mg{\mathcal{O}\left( (1-\theta)^2 \left(\delta_{(2)}^2+
\sa{(1-\theta)}
\right)+(1-\theta)^4 \left(\delta_{(4)}^4+(1-\theta)^2 \right) + \sa{(1-\theta)^2}
\label{eq: bound-e2-bis}
\right)}, 
\end{eqnarray}
as $\theta\rightarrow 1$ where \sa{we used \eqref{eq:moment2-bound} and \eqref{eq:moment4}} in the last equality. We are now ready to bound $\Vert\Delta \Vert_2$ which will in turn give us the the order of $\mE[\Vert E\Vert_2]$ with respect to $(1-\theta)$. 
\sa{Using the sub-multiplicative property of the spectral norm after computing $\norm{\Delta}_2$ from \eqref{def-Delta-matrix-2}, and then using Jensen's inequality 
to bound both $\Vert \mE[(\xi_1-\xi_*)e_2^\top]\Vert_2$ and $\Vert\mE[e_2e_2^\top]\Vert_2$,} 
we obtain 
\begin{align*}
    \Vert \Delta \Vert_2 &=\mathcal{O}\left(\Vert L \Vert_2\Vert N \Vert_2 \mE[\Vert (\xi_1-\xi_*)\bcred{e_2}^\top\Vert_2]+ \Vert N\Vert^2_2 \mE[\Vert \bcred{e_2} \bcred{e_2}^\top \Vert_2]\right)\\
    &=\mathcal{O}(\Vert L\Vert_2 \Vert N\Vert_2 \mE[\left(\Vert z_1-z_*\Vert\sa{+\Vert z_0-z_*\Vert}\right) \Vert \bcred{e_2} \Vert])+\mathcal{O}(\Vert N \Vert_2^2 \mE[\Vert \bcred{e_2}\Vert^2]),
\end{align*}
where \sa{we used the property: 
$\norm{ab^\top}_2=\norm{a}\norm{b}$ for arbitrary vectors $a$ and $b$.} \mg{From \eqref{eq: bound-e2-bis}}, we have also
$\mE[\Vert e_2\Vert]\leq \mE[\Vert e_2\Vert^2]^{1/2}=\mathcal{O}(1-\theta)$; \sa{furthermore, since both $z_1$ and $z_0$ distributed according to the stationary distribution, Lemma \ref{lem: finite-moment} implies that $\mE[\Vert z_i-z_*\Vert] \leq \mE[\Vert z_i-z_*\Vert^2]^{1/2}=\mathcal{O}\left( (1-\theta)^{1/2} \right)$ for $i=0,1$}. \mg{On the other hand, since $\tau = \mathcal{O}(1-\theta)$ and $\sigma = \mathcal{O}(1-\sigma)$, from the definition of $N$ given in \eqref{eq-M-and-N}, we also see that
$\Vert N \Vert_2=\mathcal{O}(1-\theta)$}. Combining all these bounds, 
we obtain
$$
\Vert \Delta \Vert_2= \mathcal{O}\left( (1-\theta)^{5/2}+ (1-\theta)^4\right)\sa{=\cO((1-\theta)^{5/2}).}
$$
\mg{Inserting this bound into \eqref{def-E-recall} and using the fact $\|\Delta_i\| \leq \|\Delta\|$ for $i=1,2$; we obtain
$E = \mathcal{O}\left( (1-\theta)^{5/2}
\right)$.
\sa{Then, we conclude from 
\eqref{eq: Sigma_1_ext} that}}
\begin{equation}\label{eq: Sigma_1-last}
\Sigma_1= -(1-\theta) \Gamma (W_1+W_2+W_3) \sa{+\mathcal{O}\left( (1-\theta)^{3/2} \right),}
\end{equation}
\mg{as $\theta \to 1$ where $W_i = \mathcal{O}(1)$ \sa{for $i=1,2,3$}.}
Recalling that $\Sigma_1= \mE[(z_1-z_*)^{\otimes 2}]$ by definition, inserting \eqref{eq: Sigma_1-last} into Lemma \ref{lem: char-mean-z} \mg{and using \eqref{eq-xi-vs-z}}, \mg{we conclude that \eqref{eq: inv-meas-mean} holds for the matrix $M=\frac{1}{2} \Gamma (W_1 + W_2 + W_3)$} where $M = \mathcal{O}(1)$ as $\theta\to 1$. This completes the proof. 
 }
\subsection{Proof of Theorem \ref{thm: av-seq-conv}}

\bcred{
Let $\phi: \mathbb{R}^{2d}\rightarrow \mathbb{R}$ be a measurable function with $\Vert \phi\Vert<\infty$, \sa{where $\norm{\cdot}$ denotes the weighted supremum norm defined in Theorem~\ref{thm: inv-meas-exists}}. We define the function $\psi^{(\theta)}: \mathbb{R}^{2d}\rightarrow \mathbb{R}$ such that for any initialization $\xi_0\in\mathbb{R}^{2d}$ of the Markov chain, we set }
\begin{equation}\label{def:psi^theta}
\psi^{(\theta)}(\mg{\xi_0})\sa{\triangleq} \sum_{k=0}^{\infty}(\mcR^{(\theta)}_k\bcred{\phi}(\mg{\xi_0})\bcred{-\pi_*^{\sa{(\theta)}}(\bcred{\phi}))}=\sum_{k=0}^{\infty}\left(\mE[\bcred{\phi}(\xi_k^{(\theta)})]-\mE[\bcred{\phi}(\xi^{(\theta)}_\infty)]\right).
\end{equation}



\bcred{Notice that the Theorem \ref{thm: inv-meas-exists} and the 
\sa{triangle inequality imply that there exists $C>0$ such that for all $k\geq 0$, we have}
\begin{align}
\label{eq:partial-bound}
 \Vert\sum_{i=0}^{k}\big(\mcR_i^{\theta}\phi-\pi_*^{(\theta)}(\phi)\big)\Vert \leq \sum_{i=0}^{k}\Vert \mcR_i^{(\theta)}\phi-\pi_*^{(\theta)}(\phi)\Vert\leq C \sum_{i=0}^{k}\left(\frac{2\theta}{1+\theta} \right)^{\mg{i}}\Vert \phi-\pi_*^{(\theta)}(\phi)\Vert
\end{align}
for any measurable function $\phi$ 
\sa{such that $\norm{\phi}<\infty$}. \mg{Recalling that $\frac{2\theta}{1+\theta}<1$,} this shows that the function $\xi \rightarrow \sum_{i=0}^{k}\big( \mcR_i^{(\theta)}\phi(\xi)-\pi_*^{(\theta)}\sa{(\phi)} \big)$ is uniformly bounded 
\sa{over} $k\in\mathbb{N}$.
}
Hence, the function $\psi^{(\theta)}$ is well-defined \mg{as a converging series 
with a} finite weighted supremum norm. 

\mg{We will adapt} the notation $\mcR=\mcR^{(\theta)}$ and $\bcred{\tilde{\phi}=\phi-\pi_*^{(\theta)}(\phi)}$ in the rest of the proof for simplicity. 
\sa{We next show} that $\sa{\psi=}\psi^{(\theta)}$ is the unique solution to \sa{the system:} $$(\rm{Id}-\mcR)\sa{\psi}= \tilde{\phi},\qquad \pi_*^{(\theta)}(\sa{\psi})=0,$$ where $\rm{Id}$ \sa{denotes} the identity operator, i.e., \sa{$\rm{Id}~\psi=\psi$ for any $\psi$}. \sa{Noticing that for any given $\xi_0$ and $\phi$, $\pi_*^{(\theta)}(\phi)$ is a constant and $\cR$ is a probability measure, we get $\cR \pi_*^{(\theta)}(\phi)= \pi_*^{(\theta)}(\phi)$; thus, it is immediate from the definition of $\psi^{(\theta)}$ in \eqref{def:psi^theta} that $(\rm{Id}-\mcR)\sa{\psi^{(\theta)}}= \tilde{\phi}$ holds.} \bcred{Let us 
\sa{next show that} $\pi_{*}^{(\theta)}(\psi^{(\theta)})=0$. 
\sa{Since we have} $(\rm{Id}- \mcR)\psi^{(\theta)}=\tilde{\phi}$, \sa{we also get} $(\mcR_{k-1}-\mcR_{k})\psi^{(\theta)}=\mcR_{k-1}\bcred{\tilde{\phi}}$. Hence we can write, $(\rm{Id}-\mcR_k)\psi^{(\theta)}=(\rm{Id}+\mcR_1+....+\mcR_{k-1})\bcred{\tilde{\phi}}$; therefore, in the limit $k\rightarrow \infty$, we obtain 
$$
\lim_{k\rightarrow \infty}(\psi^{(\theta)}-\mcR_k \psi^{(\theta)})= \lim_{k\rightarrow \infty} \sum_{i=0}^{k-1}\mcR_{i}\bcred{\tilde{\phi}}= \psi^{(\theta)},$$
which yields that $\lim_{k\rightarrow \infty }\mcR_k \psi^{(\theta)}=0$. On the other hand, since $\psi^{(\theta)}$ has finite \bcred{weighted} supremum norm, it follows from 
\sa{Theorem}~\ref{thm: inv-meas-exists} that $\pi_*^{(\theta)}(\psi^{(\theta)})=\lim_{k\rightarrow \infty}\mcR_k\psi^{(\theta)}=0$.
}

To see \bcred{uniqueness}, suppose there exists another function $\bar{\psi}^{(\theta)}$ such that $(Id-\mcR)\bar{\psi}^{(\theta)}=\bcred{\tilde{\phi}}$, then we can write $(\psi^{(\theta)}-\bar{\psi}^{(\theta)})= \mcR (\psi^{(\theta)}-\bar{\psi}^{(\theta)})$. Thus for all $k\in \mathbb{N}_{+}$ we obtain $(\psi^{(\theta)}-\bar{\psi}^{(\theta)})=\mcR_k (\psi^{(\theta)}-\bar{\psi}^{(\theta)})$. On the other hand, Theorem \ref{thm: inv-meas-exists} implies $(\psi^{(\theta)}-\bar{\psi}^{(\theta)})=\pi_*^{(\theta)}(\psi^{(\theta)}-\bar{\psi}^{(\theta)})=0$ which shows $\psi^{(\theta)}=\bar{\psi}^{(\theta)}$.

After presenting the properties of $\psi^{(\theta)}$, we are now ready to prove Theorem \ref{thm: av-seq-conv}. Let $\bcred{\phi}(\xi):\mathbb{R}^{2d}\rightarrow \mathbb{R}$ be a function with finite \mg{weighted} supremum norm, and define $Z_k\triangleq\frac{1}{k}\sum_{i=0}^{k-1}\bcred{\phi}(\xi_i^{(\theta)})$. \bcred{Recall that we have already shown $(\rm{Id}-\mcR_k)\psi^{(\theta)}=\sum_{i=0}^{k-1}\mcR_i\big(\bcred{\phi-\pi_*^{(\theta)}(\phi)\big)}$; hence, $Z_k$ can be written as }
\begin{align*}\label{def: Z_k}
    \mathbb{E}[Z_k]&=\bcred{\pi_*^{(\theta)}(\phi)}+ \frac{1}{k}\sum_{i=0}^{k-1}\mathbb{E}[\bcred{\big(\phi (\sa{\xi_i^{(\theta)}})-\pi_*^{(\theta)}(\phi)\big)}]\\
    &=\bcred{\pi_*^{(\theta)}(\phi)}+\frac{1}{k}\sum_{i=0}^{k-1}\bcred{\big(\mcR_i\bcred{\phi}(\xi_0)-\pi_*^{(\theta)}(\phi)\big)=\pi_*^{(\theta)}(\phi)+ \frac{\psi^{(\theta)}(\xi_0)-\mcR_k\psi^{(\theta)}(\xi_0)}{k}}.
\end{align*}
\sa{Recall that $\bar{\xi}^{(\theta)} = \mathbb{E}[\xi_\infty^{(\theta)}]$;
therefore,}
\mg{choosing the function $\phi$ as $\phi(\xi)=\xi-\bar{\xi}^{\theta}$, we have $\pi_*^{(\theta)}(\phi)=0$.} 
\sa{We have shown above that $(\rm{Id}-\mcR_k)\psi^{(\theta)}=(\rm{Id}+\mcR_1+....+\mcR_{k-1})(\phi-\pi_*^{(\theta)}(\phi))=\sum_{i=0}^{k-1}\big(\mcR_i^{\theta}\phi-\pi_*^{(\theta)}(\phi)\big)$; therefore, \eqref{eq:partial-bound} and $\pi_*^{(\theta)}(\phi)=0$ together imply that
\begin{equation}
    \frac{\psi^{(\theta)}(\xi_0)-\mcR_k\psi^{(\theta)}(\xi_0)}{k}\leq \frac{C}{k} \sum_{i=0}^{k-1}\left(\frac{2\theta}{1+\theta} \right)^{\mg{i}}\Vert \phi\Vert=\frac{1+\theta}{1-\theta} \left(1-\left(\frac{2\theta}{1+\theta}\right)^k\right)\frac{C}{k}~\norm{\phi}.
\end{equation}}%
\mg{
\sa{Since $\mathbb{E}[Z_k] = \mathbb{E}[\tilde{\xi}_k^{(\theta)}]-\bar{\xi}^{(\theta)}$; consequently, we immediately get the desired result in~\eqref{eq: gap-avg-mean-and-inv-meas-mean}.}}%
\newpage






\end{document}